\documentclass[12pt,leqno]{amsart}
\usepackage{amsmath, amssymb}
\usepackage{graphicx,color,hyperref}
\usepackage{verbatim, enumitem}
\newtheorem{theorem}{Theorem}[section]

\newtheorem{lemma}[theorem]{Lemma}
\newtheorem{proposition}[theorem]{Proposition}
\newtheorem{question}[theorem]{Question}
\newtheorem{basicquestion}[theorem]{Basic question}
\newtheorem{definition}[theorem]{Definition}
\newtheorem{remark}[theorem]{Remark}
\newtheorem{example}[theorem]{Example}

\numberwithin{equation}{section}

\def\Re{\mathop{\rm Re}\nolimits}
\def\Im{\mathop{\rm Im}\nolimits}
\def\Ker{\mathop{\rm Ker}\nolimits}

\def\Id{\mathop{\rm Id}\nolimits}
\def\End{\mathop{\rm End}\nolimits}
\def\Hom{\mathop{\rm Hom}\nolimits}
\def\Span{\mathop{\rm Span}\nolimits}
\def\Supp{\mathop{\rm Supp}\nolimits}
\def\Diff{\mathop{\rm Diff}\nolimits}
\def\rank{\mathop{\rm rank}\nolimits}
\def\pr{\mathop{\rm pr}\nolimits}
\def\sing{\mathop{\rm sing}}
\def\bu{\mathrel{\scriptscriptstyle\bullet}}

\begin{document}

\title{Algebraic embeddings of smooth almost complex structures}

\author{Jean-Pierre Demailly and Herv\'e Gaussier}
\date{}


\begin{abstract}
The goal of this work is to prove an embedding theorem for compact almost complex manifolds into complex algebraic varieties. It is shown that every almost complex structure can be realized by the transverse structure to an algebraic distribution on an affine algebraic variety, namely an algebraic subbundle of the tangent bundle. In fact, there even exist universal embedding spaces for this problem, and their dimensions grow quadratically with respect to the dimension of the almost complex manifold to embed. We give precise variation formulas for the induced almost complex structures and study the related versality conditions. At the end, we discuss the original question raised by F.Bogomolov: can one embed every compact complex manifold as a $\mathcal C^\infty$ smooth subvariety that is transverse to an algebraic foliation on a complex projective algebraic variety?
\end{abstract}

\maketitle

\section{Introduction and main results}
The goal of this work is to prove an embedding theorem for compact
almost complex mani\-folds into complex algebraic varieties. As usual,
an almost complex manifold of dimension $n$ is a pair $(X,J_X)$, where
$X$ is a real manifold of dimension $2n$ and $J_X$ a smooth section of
$\End(TX)$ such that $J_X^2=-\Id\,$; we will assume here that all data
are~$\mathcal{C}^\infty$.
\vskip1mm

Let $Z$ be a complex (holomorphic) manifold of complex dimension
$N$. Such a manifold carries a natural integrable almost complex
structure $J_Z$ (conversely, by the Newlander-Nirenberg theorem any
integrable almost complex structure can be viewed as a holomorphic
structure). Now, assume that we are given a holomorphic distribution
$\mathcal{D}$ in~$TZ$, namely a holomorphic subbundle
$\mathcal{D}\subset TZ$. Every fiber $\mathcal{D}_x$ of the
distribution is then invariant under~$J_{Z}$, i.e.\
$J_{Z}\mathcal{D}_x \subset \mathcal{D}_x$ for every $x \in Z$. Here,
the distribution $\mathcal{D}$ is not assumed to be integrable. We
recall that $\mathcal{D}$ is integrable in the sense of Frobenius
(i.e.\ stable under the Lie bracket operation) if and only if the
fibers $\mathcal{D}_x$ are the tangent spaces to leaves of a
holomorphic foliation. More precisely, $\mathcal{D}$ is integrable
if and only if the torsion operator $\theta$ of~$\mathcal{D}$,
defined~by
\begin{equation}
\theta : \mathcal{O}(\mathcal{D}) \times \mathcal{O}(\mathcal{D}) 
\longrightarrow \mathcal{O}(T Z/\mathcal{D}),
\qquad(\zeta, \eta) \longmapsto
[\zeta,\eta]~\mod~\mathcal{D}
\end{equation}
vanishes identically. As is well known, $\theta$ is skew symmetric in 
$(\zeta,\eta)$ and can be viewed as a holomorphic section of the bundle 
$\Lambda^2\mathcal{D}^*\otimes(T Z/\mathcal{D})$.
\vskip1mm
Let $M$ be a real submanifold of $Z$ of class $ \mathcal C^\infty$ and
of real dimension $2n$ with $n < N$. We say that $M$ is
\emph{transverse to $\mathcal{D}$} if for every $x \in M$ we have
\begin{equation}
T_xM \oplus \mathcal{D}_x = T_x Z.
\end{equation}
We could in fact assume more generally that the distribution
$\mathcal{D}$ is singular, i.e.\ given by a certain saturated subsheaf
$\mathcal{O}(\mathcal{D})$ of $\mathcal{O}(TZ)$ (``saturated'' means
that the quotient sheaf $\mathcal{O}(TZ)/\mathcal{O}(\mathcal{D})$ has
no torsion). Then $\mathcal{O}(\mathcal{D})$ is actually a subbundle
of $TZ$ outside an analytic subset $\mathcal{D}_{\sing}\subset Z$ of
codimension${}\geq 2$, and we further assume in this case that $M\cap
\mathcal{D}_{\sing}=\emptyset$.
\vskip 1mm
When $M$ is transverse to $\mathcal{D}$, one gets a natural
$\mathbb{R}$-linear isomorphism 
\begin{equation}\label{tangent-isom}
T_xM\simeq T_xZ/\mathcal{D}_x
\end{equation}
at every point $x\in M$. Since $TZ/\mathcal{D}$ carries a structure of
holomorphic vector bundle (at least over
$Z\smallsetminus\mathcal{D}_{\sing}$), the complex structure $J_Z$
induces a complex structure on the quotient and therefore, through
the above isomorphism (\ref{tangent-isom}), an almost complex
structure $J_M^{Z,\mathcal{D}}$ on~$M$.
\vskip1mm
Moreover, when $\mathcal{D}$ is a foliation (i.e.\
$\mathcal{O}(\mathcal{D})$ is an integrable subsheaf of $\mathcal{O}(TZ)$), then
$J_M^{Z,\mathcal{D}}$ is an \emph{integrable} almost complex structure
on~$M$. Indeed, such a foliation is realized near any regular point
$x$ as the set of fibers of a certain submersion: there exists an open
neighborhood $\Omega$ of $x$ in $Z$ and a holomorphic submersion
$\sigma:\Omega \to\Omega'$ to an open subset
$\Omega'\subset\mathbb{C}^n$ such that the fibers of $\sigma$ are
the leaves of $\mathcal{D}$ in~$\Omega$. We can take $\Omega$ to be a
coordinate open set in $Z$ centered at point~$x$ and select
coordinates such that the submersion is expressed as the first
projection $\Omega\simeq\Omega'\times\Omega''\to\Omega'$ with respect
to $\Omega'\subset\mathbb{C}^n$, $\Omega''\subset\mathbb{C}^{N-n}$,
and then $\mathcal{D}$, $TZ/\mathcal{D}$ are identified with the
trivial bundles $\Omega\times(\{0\}\times\mathbb{C}^{N-n})$ and
$\Omega\times\mathbb{C}^{n}$. The restriction
$$\sigma_{M\cap\Omega}:M\cap\Omega\subset \Omega
\mathop{\longrightarrow}^\sigma\Omega'$$
provides $M$ with holomorphic coordinates on~$M\cap\Omega$, and it is
clear that any other local trivialization of the foliation on a
different chart
$\widetilde\Omega=\widetilde\Omega'\times\widetilde\Omega''$ would
give coordinates that are changed by local biholomorphisms
$\Omega'\to\widetilde\Omega'$ in the intersection
$\Omega\cap\widetilde\Omega$, thanks to the holomorphic character
of~$\mathcal{D}$. Thus we directly see in that case that
$J_M^{Z,\mathcal{D}}$ comes from a holomorphic structure on~$M$.

More generally, we say that $f:X\hookrightarrow Z$ is a transverse
embedding of a smooth real manifold $X$ in $(Z,\mathcal{D})$ if $f$ is
an embedding and $M=f(X)$ is a transverse submanifold of $Z$, namely
if $f_*T_xX\oplus \mathcal{D}_{f(x)}=T_{f(x)}Z$ for every point $x\in
X$ (and $f(X)$ does not meet $\mathcal{D}_{\sing}$ in case there are
singularities). One then gets a real isomorphism $TX\simeq
f^*(TZ/\mathcal{D})$ and therefore an almost complex structure on $X$
(for this it would be enough to assume that $f$ is an immersion, but
we will actually suppose that $f$ is an embedding here). We denote by $J_f$ the almost complex structure $J_f:=f^*(J_{f(X)}^{Z,\mathcal D})$.
\vskip1mm

In this work, we are interested in the problem of embedding a compact complex or almost complex manifold $X$ into a {\it projective algebraic} manifold $Z$, {\it transversally} to an algebraic distribution $\mathcal{D}\subset TZ$. We will also make use of the concept of Nash algebraicity. Recall that a {\it Nash algebraic map} $g:U\to V$ between connected open (i.e.\ metric open) sets $U,\,V$ of algebraic manifolds $Y,\,Z$ is a map whose graph is a connected component of the intersection of $U\times V$ with an algebraic subset of $Y\times Z$. We say that a holomorphic foliation ${\mathcal F}$ of codimension $n$ on $U$ is {\it algebraic} (resp.\ {\it holomorphic}, {\it Nash algebraic}) if the associated distribution $Y\supset U\to{\rm Gr}(TY,n)$ into the Grassmannian bundle of the tangent bundle is given by an algebraic (resp.\ holomorphic, Nash algebraic) morphism. The following very interesting question was investigated about 20 years
ago by F.~Bogomolov \cite{bogomolov}.

\begin{basicquestion}\label{bogo-question} 
Given an integrable complex structure $J$
on a compact manifold $X$, can one realize $J$, as described above,
by a transverse embedding $f:X\hookrightarrow Z$ into a projective manifold
$(Z,\mathcal{D})$ equipped with an algebraic foliation~$\mathcal{D}$,
in such a way that $f(X)\cap \mathcal{D}_{\sing}=\emptyset$ and $J=J_f\,?$
\end{basicquestion}

There are indeed many examples of K\"ahler and non K\"ahler compact
complex manifolds which can be embedded in that way (the case of
projective ones being of course trivial): tori, Hopf and Calabi-Eckman 
manifolds, and more generally all manifolds given by the LVMB construction
(see Section~\ref{section-foliations}). Strong indications exist
that every compact complex manifold should be embeddable as a
smooth submanifold transverse to an algebraic foliation on a complex
projective variety, see section~\ref{section-bogomolov} . We prove here
that the corresponding statement in the almost complex category actually 
holds -- provided that non integrable distributions are considered
rather than foliations.  In fact, there are even ``universal
solutions'' to this problem.

\begin{theorem}\label{theorem-univ-embed}
For all integers $n\ge 1$ and $k\ge 4n$, there exists a complex affine 
algebraic manifold $Z_{n,k}$ of dimension $N=2k+2(k^2+n(k-n))$ possessing
a  real structure $($i.e.\ an anti-holomorphic algebraic involution$)$ 
and an algebraic distribution $\mathcal{D}_{n,k}\subset TZ_{n,k}$ of 
codimension $n$, for which
every compact $n$-dimensional almost complex manifold $(X,J)$ admits
an embedding $f:X\hookrightarrow Z^{\mathbb{R}}_{n,k}$ transverse to 
$\mathcal{D}_{n,k}$ and contained in the real part of $Z_{n,k}$,
such that $J=J_f$.
\end{theorem}

\begin{remark} {\rm
To construct $f$ we first embed $X$ differentiably into
$\mathbb{R}^k$, $k\ge 4n$, by the Whitney embedding theorem
\cite{whitney}, or its generalization due to \cite{tognoli}. Once the
embedding of the underlying differentiable manifold has been fixed,
the transverse embedding $f$ depends in a simple algebraic way on the
almost complex structure~$J$ given on~$X$, as one can see from our
construction (see Section 4).}
\end{remark}

The choice $k=4n$ yields the explicit embedding dimension 
$N=38n^2 + 8n$ (we will see that a quadratic bound $N=O(n^2)$ 
is optimal, but the above explicit value could perhaps be improved).
Since $Z=Z_{n,k}$ and $\mathcal{D}=\mathcal{D}_{n,k}$ are algebraic
and $Z$ is affine, one can further compactify $Z$ into a complex
projective manifold $\overline Z$, and extend $\mathcal{D}$ into a
saturated subsheaf $\overline{\mathcal{D}}$ of $T\overline Z$.  In
general such distributions $\overline{\mathcal{D}}$ will acquire
singularities at infinity, and it is unclear whether one can achieve
such embeddings with $\overline{\mathcal{D}}$ non singular on
$\overline Z$, if at all possible.
\smallskip

Next, we consider the case of a compact almost complex symplectic manifold
$(X,J,\omega)$ where the symplectic form $\omega$ is assumed to be
$J$-compatible, i.e.\ $J^*\omega=\omega$ and \hbox{$\omega(\xi,J\xi)>0$}. 
By a theorem of Tischler 
\cite{tischler}, at least under the assumption that the De Rham cohomology
class $\{\omega\}$ is integral, we know that there exists a smooth embedding
$g:X\hookrightarrow\mathbb{CP}^s$ such that $\omega=g^*\omega_{\rm FS}$
is the pull-back of the standard Fubini-Study metric on $\omega_{\rm FS}$
on $\mathbb{CP}^s$. A natural problem is whether the symplectic structure
can be accommodated simultaneously with the almost complex structure
by a transverse embedding. Let us introduce the following definition.

\begin{definition} Let $(Z,\mathcal{D})$ be a complex manifold equipped with
a holomorphic distribution. We say that a closed semipositive $(1,1)$-form
$\beta$ on $Z$ is a transverse K\"ahler structure if the kernel of
$\beta$ is contained in $\mathcal{D}$, in other terms, if $\beta$ induces
a K\"ahler form on any germ of complex submanifold transverse to
$\mathcal{D}$.
\end{definition}

Using an effective version of Tischler's theorem stated by Gromov
\cite{gromov}, we prove~:

\begin{theorem}\label{theorem-univ-embed-symplectic}
For all integers $n\ge 1$, $b\ge 1$ and $k\ge 2n+1$, there exists a
complex projective algebraic manifold $Z_{n,b,k}$ of dimension
$N=2bk(2bk+1)+2n(2bk-n))$, equipped with a real structure and an
algebraic distribution $\mathcal{D}_{n,b,k}\subset TZ_{n,b,k}$ of
codimension $n$, for which every compact $n$-dimensional almost
complex symplectic manifold $(X,J, \omega)$ with second Betti number
$b_2\le b$ and a $J$-compatible symplectic form $\omega$ admits an
embedding $f:X\hookrightarrow Z^{\mathbb{R}}_{n,b,k}$ transverse to
$\mathcal{D}_{n,b,k}$ and contained in the real part of $Z_{n,k}$,
such that $J=J_f$ and $\omega=f^*\beta$ for some transverse K\"ahler
structure $\beta$ on $(Z_{n,b,k},\mathcal{D}_{n,b,k})$.
\end{theorem}

In section~\ref{section-bogomolov}, we discuss Bogomolov's
conjecture for the integrable case. We first prove the following weakened version, which can be seen as a form of ``algebraic embedding'' for arbitrary compact complex manifolds.

\begin{theorem}\label{bogomolov-thm}
For all integers $n\ge 1$ and $k\ge 4n$, let $(Z_{n,k},\mathcal{D}_{n,k})$
be the affine algebraic manifold equipped with the algebraic distribution
$\mathcal{D}_{n,k}\subset TZ_{n,k}$ introduced in 
Theorem~\ref{theorem-univ-embed}. Then, for every compact $n$-dimensional 
$($integrable$)$ complex manifold $(X,J)$, there exists an embedding 
$f:X\hookrightarrow Z^{\mathbb{R}}_{n,k}$ transverse to $\mathcal{D}_{n,k}$,
contained in the real part of $Z_{n,k}$, such that 
\begin{itemize}
\item[\rlap{\kern-1cm\rm(i)}] $J=J_f$ and $\overline{\partial}_J f$ is injective$\,;$
\vskip4pt
\item[\rlap{\kern-1cm\rm(ii)}] $\Im(\overline{\partial}_{J}f)$ is contained in the isotropic locus $I_{{\mathcal D}_{n,k}}$ of the torsion operator $\theta$ of ${\mathcal D}_{n,k}$, the intrinsically defined algebraic locus in the Grassmannian bundle \hbox{${\rm Gr}({\mathcal D}_{n,k},n)\to Z_{n,k}$} of complex $n$-dimensional subspaces in ${\mathcal D}_{n,k}$, consisting of those subspaces $S$ such that $\theta_{|S\times S}=0$.
\end{itemize}
\end{theorem}

The inclusion condition (ii) $\Im(\overline{\partial}_{J}f)\subset I_{{\mathcal D}_{n,k}}$ is in fact necessary and sufficient for the integrability of $J_f$. 
\vskip3mm

In section~\ref{nash-approx-foliations}, we investigate the original Bogomolov conjecture and ``reduce it'' to a statement concerning approximations of holomorphic foliations. The flavor of the statement is that holomorphic objects (functions, sections of algebraic bundles, etc) defined on a polynomially convex open set of ${\mathbb C}^n$ can always be approximated by polynomials or algebraic sections. Our hope is that this might be true also for the approximation of holomorphic foliations by Nash algebraic ones. In fact, we obtain the following conditional statement.

\begin{proposition}\label{runge-prop} Assume that holomorphic foliations can be approximated by Nash algebraic foliations uniformly on compact subsets of any polynomially convex open subset of ${\mathbb C}^N$. Then every compact complex manifold can be approximated by compact complex manifolds that are embeddable in the sense of Bogomolov in foliated projective manifolds.
\end{proposition}

In the last section~\ref{categorical}, we briefly discuss a ``categorical'' viewpoint in which the above questions have a nice interpretation. The authors wish to thank the anonymous referee for his careful reading of the paper, and for a number of useful remarks and demands of clarification.

\section{Transverse embeddings to foliations}\label{section-foliations}

We consider the situation described above, where $Z$ is a complex
$N$-dimensional manifold equipped with a holomorphic distribution
$\mathcal{D}$. More precisely, let $X$ be a compact real manifold of
class $\mathcal C^\infty$ and of real dimension $2n$ with $n < N$. We
assume that there is an embedding $f:X \hookrightarrow Z$ that is
transverse to $\mathcal{D}$, namely that $f(X)\cap
\mathcal{D}_{\sing}=\emptyset$ and
\begin{equation}
f_*T_xX\oplus \mathcal{D}_{f(x)} = T_{f(x)}Z
\end{equation}
at every point $x \in X$. Here $\mathcal{D}_{f(x)}$ denotes the fiber
at $f(x)$ of the distribution $\mathcal{D}$. As explained in
section~1, this induces a $\mathbb{R}$-linear isomorphism $f_* : TX
\rightarrow f^*(TZ/\mathcal{D})$, and from the complex structures of
$TZ$ and $\mathcal{D}$ we get an almost complex structure
$f^*J_{f(X)}^{Z,\mathcal{D}}$ on $TX$ which we will simply denote by
$J_f$ here. Next, we briefly investigate the effect of isotopies.

\begin{definition}An isotopy of smooth transverse embeddings of $X$ into 
$(Z,\mathcal{D})$ is by definition a family $f_t:X\to Z$ of embeddings for
$t\in[0,1]$, such that the map $F(x,t)=f_t(x)$ is smooth on $X\times[0,1]$
and $f_t$ is transverse to $\mathcal{D}$ for every $t\in[0,1]$.
\end{definition}

We then get a smooth variation $J_{f_t}$ of almost complex structures on~$X$.
When $\mathcal{D}$ is integrable (i.e.\ a holomorphic foliation), these 
structures are integrable and we have the following simple but remarkable fact.

\begin{proposition}Let $Z$ be a compact complex manifold equipped with
a holomorphic foliation $\mathcal{D}$ and let $f_t:X\to Z$,
$t\in[0,1]$, be an isotopy of transverse embeddings of a compact
smooth real manifold. Then all complex
structures $(X,J_{f_t})$ are biholomorphic to $(X,J_{f_0})$ through a 
smooth variation of diffeomorphisms in $\Diff_0(X)$, the identity 
component of the group $\Diff(X)$ of diffeomorphisms of~$X$.
\end{proposition}

\begin{proof} By an easy connectedness argument, it is enough to
produce a smooth variation of biholomorphisms $\psi_{t,t_0}:(X,J_{f_{t_0}})$ to
$(X,J_{f_t})$ when $t$ is close to $t_0$, and then extend these to all $t,t_0\in
[0,1]$ by the chain rule. Let $x\in X$. Thanks
to the local triviality of the foliation at $z_0=f_{t_0}(x)\in Z\smallsetminus
\mathcal{D}_{\sing}$, $\mathcal{D}$ is locally near $x$ the family of fibers
of a holomorphic submersion $\sigma:Z\supset\Omega\to\Omega'\subset
\mathbb{C}^n$ defined on a neighborhood $\Omega$ of~$z_0$. Then 
$\sigma\circ f_t:X\supset f_t^{-1}(\Omega)\to \Omega'$ is
by definition a local biholomorphism from $(X,J_{f_t})$ to $\Omega'$ 
(endowed with the standard complex structure of $\mathbb{C}^n$).
Now, $\psi_{t,t_0}=(\sigma\circ f_t)^{-1}\circ (\sigma\circ f_{t_0})$
defines a local biholomorphism from $(X,J_{f_{t_0}})$ to $(X,J_{f_t})$
on a small neighborhood of $x$, and these local biholomorphisms glue
together into a global one when $x$ and $\Omega$ vary 
(this biholomorphism consists of ``following the leaf of 
$\mathcal{D}$'' from the position $f_{t_0}(X)$ to the position $f_t(X)$ of
the embedding). Clearly $\psi_{t,t_0}$~depends smoothly
on $t$ and satisfies the chain rule $\psi_{t,t_0}\circ \psi_{t_0,t_1}=
\psi_{t,t_1}$.
\end{proof}

Therefore when $\mathcal D$ is a foliation, to any triple
$(Z,\mathcal{D},\alpha)$ where $\alpha$ is
an isotopy class of transverse embeddings \hbox{$X\to Z$}, one can attach a
point in the Teichm\"uller space $\mathcal{J}^{\rm int}(X)/\Diff_0(X)$
of integrable almost complex structures modulo biholomorphisms diffeotopic
to identity. The question raised by Bogomolov can then be stated more
precisely$\,$:

\begin{question} For any compact complex manifold $(X,J)$, does there 
exist a triple $(Z,\mathcal{D},X,\alpha)$ formed by a smooth complex 
projective variety $Z$, an algebraic foliation $\mathcal{D}$ on $Z$ and 
an isotopy class $\alpha$ of transverse embeddings $X\to Z$, such that
$J=J_f$ for some $f\in\alpha\,?$
\end{question}

The isotopy class of embeddings $X\to Z$ in a triple 
$(Z,\mathcal{D},\alpha)$ provides
some sort of ``algebraicization'' of a compact complex manifold, 
in the following sense$\,$:
\begin{lemma}\label{alg-lem}
There is an atlas of $X$ such that the transition functions are
solutions of algebraic linear equations (rather than plain algebraic
functions, as would be the case for usual algebraic varieties).
In this setting, the isotopy classes $\alpha$ 
are just ``topological classes'' belonging to a discrete countable set.
\end{lemma}

This set can be infinite as one already sees for real linear
embeddings of a real even dimensional torus $X=(\mathbb{R}/\mathbb{Z})^{2n}$
into a complex torus $Z=\mathbb{C}^N/\Lambda$ equipped with a linear 
foliation $\mathcal{D}$.

\begin{proof}We first cover $Z\smallsetminus
\mathcal{D}_{\sing}$ by a countable family of coordinate open sets
$\Omega_\nu\simeq\Omega'_\nu\times\Omega''_\nu$ such that the first projections
$\sigma_\nu:\Omega_\nu\to\Omega'_\nu\subset\mathbb{C}^n$ define the foliation.
We assume here $\Omega'_\nu$ and $\Omega''_\nu$ to be balls of 
sufficiently small radius, so that all fibers $z'\times\Omega''_\nu$
are geodesically convex with respect to a given hermitian metric on the
ambient manifold~$Z$, and the geodesic segment joining any two points in
those fibers is unique (of course, we mean here geodesics relative
to the fibers -- standard results of differential geometry guarantee that
sufficiently small coordinate balls will satisfy this property).
Then any nonempty intersection $\bigcap z'_j\times\Omega''_{\nu_j}$
of the fibers from various coordinate sets is still connected and 
geodesically convex.
We further enlarge the family with all smaller balls whose centers have
coordinates in $\mathbb{Q}[i]$ and radii in $\mathbb{Q}_+$, so that 
arbitrarily fine cove\-rings can be extracted from the family. 
A transverse embedding $f:X\to Z$ is characterized  by its image 
$M=f(X)$ up to right composition with an element $\psi\in\Diff(X)$, and 
thus, modulo isotopy, up to an element in the countable mapping 
class group $\Diff(X)/\Diff_0(X)$.
The image $M=f(X)$ is itself given by a finite collection of graphs
of maps $g_\nu:\Omega'_\nu\to\Omega''_\nu$ that glue together, for a certain
finite subfamily of coordinate sets $(\Omega_\nu)_{\nu\in I}$ extracted 
from the initial countable fa\-mily. However, any two such 
transverse submanifolds $(M_k)_{k=0,1}$ and associated collections of 
graphs $(g_{k,\nu})$ defined on the same finite subset $I$ are isotopic$\,$:
to see this, assume e.g.\ $I=\{1,2,\ldots,s\}$ and
fix even smaller products of balls $\widetilde\Omega_\nu\simeq
\widetilde\Omega'_\nu\times\widetilde\Omega''_\nu\Subset
\Omega_\nu$ still covering $M_0$ and $M_1$, and~a~cut-off function
$\theta_\nu(z')$ equal to $1$ on $\widetilde\Omega_\nu'$ and with support
in~$\Omega'_\nu$. Then we construct isotopies 
$(F_{t,k})_{t\in[0,1]}:M_0\to M_{t,k}$ step by step, for $k=1,2,\ldots,s$, 
by taking inductively graphs
of maps $(G_{t,k,\nu})_{t\in[0,1],\,k=1,...,s,\,\nu\in I}$ such that
\begin{eqnarray*}
&&M_{t,1}\quad\hbox{given by}~~
\left\{\kern-18pt\begin{matrix}
&&
G_{t,1,1}(z')=\gamma\big(t\theta_1(z')\,;\,g_{0,1}(z'),g_{1,1}(z')\big)
\quad\hbox{on}~~\Omega'_1,\hfill\\
\noalign{\vskip5pt}
&&G_{t,1,\nu}(z')=g_{0,\nu}(z')\quad\hbox{on}~~\sigma_\nu\big(
\Omega_\nu\smallsetminus(\Supp(\theta_1)\times\widetilde\Omega''_1)\big),~~
\nu\ne 1,\hfill
\end{matrix}\right.\\
\noalign{\vskip5pt}
&&M_{t,k}\quad\hbox{given by}~~
\left\{\kern-18pt\begin{matrix}
&&G_{t,k,k}(z')=\gamma\big(t\theta_k(z')\,;\,G_{t,k-1,k}(z'),
G_{t,k-1,k}(z')\big)\quad\hbox{on}~~\Omega'_k,\hfill\\
\noalign{\vskip5pt}
&&G_{t,k,\nu}(z')=G_{t,k-1,\nu}(z')\quad\hbox{on}~~\sigma_\nu\big(
\Omega_\nu\smallsetminus(\Supp(\theta_k)\times\widetilde\Omega''_k)\big),~~
\nu\ne k,\hfill
\end{matrix}\right.
\end{eqnarray*}
where $\gamma(t\,;\,a'',b'')$ denotes the geodesic segment between
$a''$ and $b''$ in each fiber $z'\times\Omega''_\nu$. By~construction, we have
$M_{0,k}=M_0$ and $M_{1,k}\cap U_k=M_1\cap U_k$ on 
\hbox{$U_k=\smash{\widetilde\Omega_1\cup\ldots\cup\widetilde\Omega_k}$}, thus
$f_t:=F_{t,s}:M_0\to M_t$ is a transverse isotopy between 
$M_0$ and $M_1$. Therefore, we have at most as many isotopy classes as the 
cardinal of the mapping class group, times the cardinal of the set of finite
subsets of a countable~set, which is still countable.
\end{proof}
\vskip1mm
Of course, when $\mathcal{D}$ is non integrable, the almost complex structure
$J_{f_t}$ will in general vary under isotopies. One of the goals of the next 
sections is to investigate this phenomenon, but in this section we study
further some integrable examples.

\begin{example}[Complex tori] {\rm Let $X=\mathbb{R}^{2n}/\mathbb{Z}^{2n}$
be an even dimensional real torus and \hbox{$Z=\mathbb{C}^N/\Lambda$} a 
complex torus where $\Lambda\simeq\mathbb{Z}^{2N}$ is a lattice 
of~$\mathbb{C}^N$, $N>n$.
Any complex vector subspace $\mathcal{D}\subset\mathbb{C}^N$ of codimension
$n$ defines a linear foliation on $Z$ (which may or may not have closed 
leaves, but for $\mathcal{D}$ generic, the leaves are everywhere dense). Let
$f:X\to Z$ be a linear embedding transverse to $\mathcal{D}$. Here, there 
are countably many distinct isotopy classes of such linear embeddings, 
in fact up to a translation, $f$ is induced by a $\mathbb{R}$-linear map
$u:\mathbb{R}^{2n}\to\mathbb{C}^N$ that sends the standard basis 
$(e_1,\ldots,e_{2n})$ of $\mathbb{Z}^{2n}$ to a unimodular 
system of $2n$ $\mathbb{Z}$-linearly independent vectors
$(\varepsilon_1,\ldots,\varepsilon_{2n})$ of $\Lambda$. Such 
$(\varepsilon_1,\ldots,\varepsilon_{2n})$ can be chosen to 
generate any $2n$-dimensional
$\mathbb{Q}$-vector subspace $V_\varepsilon$ of $\Lambda\otimes\mathbb{Q}
\simeq\mathbb{Q}^{2N}$, thus the permitted directions for $V_\varepsilon$ 
are dense, and for most of them $f$ is indeed transverse to $\mathcal{D}$. 
For a transverse linear embedding, we get a $\mathbb{R}$-linear 
isomorphism $\tilde u:\mathbb{R}^{2n}\to\mathbb{C}^N/\mathcal{D}$, 
and the complex 
structure $J_f$ on $X$ is precisely the one induced by that isomorphism
by pulling-back the standard complex structure on the quotient.
For $N\ge 2n$, we claim that all possible translation invariant 
complex structures on $X$ are obtained. In fact, we can then choose
the lattice vector images $\varepsilon_1,\ldots,\varepsilon_{2n}$
to be $\mathbb{C}$-linearly independent, so that the map
$u:\mathbb{Z}^{2n}\to\Lambda$, $e_j\mapsto\varepsilon_j$ extends into an
injection $v:\mathbb{C}^{2n}\to \mathbb{C}^N$. Once this is done,
the isotopy class of embedding is determined, and a translation invariant 
complex structure $J$ on $X$ is given by a direct sum decomposition
$\mathbb{C}^{2n}=S\oplus\overline S$ with $\dim_{\mathbb{C}}S=n$
(and $\overline S$ the complex congugate of $S$). What we need is
that the composition $\tilde v:\mathbb{C}^{2n}\to\mathbb{C}^N
\to\mathbb{C}^N/\mathcal{D}$ defines
a $\mathbb{C}$-linear isomorphism of $S$ onto $\tilde v(S)\subset
\mathbb{C}^N/\mathcal{D}$ and $\tilde v(\overline S)=\{0\}$, i.e.\
$\mathcal{D}\supset v(\overline S)$ and $\mathcal{D}\cap v(S)=\{0\}$.
The solutions are obtained by taking $\mathcal{D}=v(\overline S)\oplus H$,
where $H$ is any supplementary subspace of $v(S\oplus\overline S)$ in
$\mathbb{C}^N$ (thus the choice of $\mathcal{D}$ is unique if $N=2n$, and 
parametrized by an affine chart of a Grassmannian $G(N-n,N-2n)$ if $N>2n$).
Of course, we can take here $Z$ to be an Abelian variety -- even a 
simple Abelian variety if we wish.}
\end{example}

\begin{example}[LVMB manifolds] {\rm We refer to L\'opez de Medrano-Verjovsky 
\cite{LoV97}, Meersseman \cite{meersseman} and Bosio \cite{bosio} for 
the original constructions, and sketch here the more general definition
given in \cite{bosio} (or rather an equivalent one, with very minor changes
of notation). Let $m\geq 1$ and $N\ge 2m$ be integers, and let 
${\mathcal E}={\mathcal E}_{m,N+1}$ be a non empty set of subsets of 
cardinal $2m+1$ of $\{0,1,\ldots,N\}$. For
$J\in{\mathcal E}$, define $U_J$ to be the open set of
points $[z_0:\ldots:z_N]\in{\mathbb CP}^N$ such that 
$z_j\ne 0$ for $j\in J$ and $U_{\mathcal E}=\bigcup_{J\in{\mathcal E}}U_J$. Then, consider the action of ${\mathbb C}^m$ on $U_{\mathcal E}$ given by
$$
w\cdot[z_0:\ldots:z_N]=\big[e^{\ell_0(w)}z_0:\ldots:e^{\ell_N(w)}z_N\big]
$$
where $\ell_j\in({\mathbb C}^m)^*$ are complex linear forms ${\mathbb C}^m\to{\mathbb C}$, $0\le j\le N$. Then Bosio proves (\cite{bosio}, Th\'eor\`eme 1.4), 
that the space of orbits $X=U_{\mathcal E}/{\mathbb C}^m$ is a compact complex manifold of dimension $n=N-m$ if and only if the following two combinatorial conditions are met:
\begin{itemize}
\item[\rlap{\kern-1cm\rm(i)}] for any two sets $J_1,\,J_2\in{\mathcal E}$, the
convex envelopes in $({\mathbb C}^m)^*$ of $\{\ell_j\}_{j\in J_1}$ and
$\{\ell_j\}_{j\in J_2}$ overlap on some non empty open set.
\vskip2mm
\item[\rlap{\kern-1cm\rm(ii)}] for all $J\in{\mathcal E}$ and 
$k\in\{0,\ldots,N\}$, there exists $k'\in J$ such that 
$(J\smallsetminus\{k'\})\cup\{k\}\in{\mathcal E}$.
\end{itemize}
The above action can be described in terms of $m$ pairwise commuting 
Killing vector fields of the action of ${\rm PGL}(N+1,{\mathbb C})$
on $\mathbb{CP}^N$, given by
$$
\zeta_j=\sum_{k=0}^N\lambda_{jk}z_k\frac{\partial}{\partial z_k},\qquad
\lambda_{jk}=\frac{\partial \ell_k}{\partial w_j},\quad 1\le j\le m.
$$
These vector fields generate a foliation ${\mathcal F}$ of dimension
$m$ on $\mathbb{CP}^N$ that is non singular over~$U_{\mathcal E}$.
Under the more restrictive condition defining LVM manifolds, it
follows from \cite{meersseman} that $X$ can be embedded as a smooth
compact real analytic submanifold $M$ in $U_{\mathcal E}$ that is
transverse to~${\mathcal F}\,$; such a submanifold $M$ is realized as
the transverse intersection of hermitian quadrics
\hbox{$\sum_{0\le k\le N}\lambda_{j,k}|z_k|^2=0$}, $1\le j\le m$ (this
actually yields $2m$ real conditions by taking real and imaginary
parts). In the more general case of LVMB manifolds, Bosio has observed
that $X$ can also be embedded smoothly in
$U_{\mathcal E}\subset \mathbb{CP}^N$ (see \cite{bosio}, Prop.~2.3 and
discussion thereafter; and also \cite{bosio-meer}, Part~III,
section~12).  }
\end{example}

\section{Deformation of transverse embeddings}\label{deform-section}

Let $f : X \to (Z,\mathcal D)$ be a transverse embedding. Then
$J_f:=f^*(J_{f(X)}^{Z,\mathcal D})$ defines an almost complex
structure on $X$. We give in this section sufficient conditions on the
embedding $f$ that ensure that small deformations of $J_f$, in a
suitable space of almost complex structures on $X$, are given as
$J_{\tilde{f}}$ where $\tilde{f}$ are small deformations of $f$ in a
suitable space of transverse embeddings of $X$ into $(Z,\mathcal{D})$.
Since the implicit function theorem will be needed, we have to
introduce various spaces of $\mathcal{C}^r$
mappings. For any $r\in [1,\infty]$, we consider the 
group $\Diff^r(X)$ of diffeomorphisms of $X$ of class~$\mathcal{C}^r$, and
$\Diff^r_0(X)\subset\Diff^r(X)$ the subgroup of diffeomorphisms
diffeotopic to identity. When $r=s+\gamma$ is not an integer, 
$s=\lfloor r\rfloor$, then $\mathcal{C}^r$ denotes the H\"older space of maps of class $\mathcal{C}^s$ with derivatives of order $s$ that are
H\"older continuous with exponent $\gamma$. Similarly, we consider the space
$\mathcal{C}^r(X, Z)$ of $\mathcal{C}^r$ mappings \hbox{$X\to Z$} equipped
with $\mathcal{C}^r$ convergence topology (of course, in $\Diff^r(X)$,
the topology also requires convergence of sequences $f_\nu^{-1}$). If 
$Z$ is Stein, there exists a biholomorphism \hbox{$\Phi:TZ\to Z\times Z$}
from a neighborhood of the zero section of $TZ$ to a neighborhood of the
diagonal in $Z\times Z$, such that $\Phi(z,0)=z$ and
$d_\zeta\Phi(z,\zeta)_{|\zeta=0}=\Id$ on $T_zZ$. When $Z$ is embedded
in $\mathbb{C}^{N'}$ for some $N'$, such a map can be obtained by
taking $\Phi(z,\zeta)=\rho(z+\zeta)$, where $\rho$ is a local
holomorphic retraction $\mathbb{C}^{N'}\to Z$ and $T_zZ$ is identified
to a vector subspace of $\mathbb{C}^{N'}$. In general (i.e.\ when $Z$
is not necessarily Stein), one can still find a $\mathcal{C}^\infty$
or even real analytic map $\Phi$ satisfying the same conditions, by 
taking e.g.\ $\Phi(z,\zeta)=(z,\exp_z(\zeta))$, where $\exp$ is the Riemannian
exponential map of a real analytic hermitian metric on~$Z\,$; actually, we
will not need $\Phi$ to be holomorphic in what follows.

\begin{lemma}\label{banach-isom1} For $r\in [1,\infty[\,$, 
$\mathcal{C}^r(X, Z)$ is a Banach manifold whose tangent space at a point
$f:X\to Z$ is $\mathcal{C}^r(X,f^*TZ)$, and $\Diff^r_0(X)$  is
a ``Banach Lie group'' with ``Lie algebra'' $\mathcal{C}^r(X,TX)$ $[$the
quotes meaning that the composition law is not real analytic as one would
expect, but merely continuous and differentiable at $\Id_X$, though 
the underlying manifold is indeed a Banach manifold$\,]$.
\end{lemma}
Let us also point out that if the composition of $\mathcal{C}^r$ maps 
is merely $\mathcal{C}^{r^2}$ for $0 < r < 1$, it is actually a 
$\mathcal{C}^r$ map for $r \geq 1$.  

\begin{proof} A use of the map $\Phi$ allows us to parametrize small
deformations of the embedding $f$ as $\widetilde
f(x)=\Phi(f(x),u(x))$ [or equivalently
$u(x)=\Phi^{-1}(f(x),\widetilde f(x))\,$], where $u$ is a smooth
sufficiently small section of $f^*TZ$. This parametrization is
one-to-one, and $\widetilde f$ is $\mathcal{C}^r$ if and only if $u$
is $\mathcal{C}^r$ (provided $f$ is). The argument is similar, and
very well known indeed, for~$\Diff^r_0(X)$.
\end{proof}

Now, let $\mathcal{J}^r(X)$ denote the space of almost complex structures of
class $\mathcal{C}^r$ on~$X$. For $1 \leq r < +\infty$, this is a Banach
mani\-fold whose tangent space at a point $J$ is the space of sections
$h\in\mathcal{C}^r(X,\End_{\overline{\mathbb{C}}}(TX))$ satisfying
$J\circ h+h\circ J=0$ (namely conjugate $\mathbb{C}$ linear
endomorphisms of~$TX$). There is a natural right action of
$\Diff^r_0(X)$ on $\mathcal{J}^{r-1}(X)$ defined by
$$
(J,\psi)\mapsto \psi^*J,\qquad \psi^*J(x)=d\psi(x)^{-1}\circ 
J(\psi(x))\circ d\psi(x).
$$
As is well-known and as a standard calculation shows, the differential of 
$\psi\mapsto \psi^*J$ at $\psi=\Id_X$ is closely related to the 
$\overline\partial_J$ operator
$$
\overline\partial_J:\mathcal{C}^r(X,TX)\longrightarrow
\mathcal{C}^{r-1}(X,\Lambda^{0,1}TX^*\otimes TX^{1,0})=
\mathcal{C}^{r-1}(X,\End_{\overline{\mathbb{C}}}(TX)),
$$
namely it is given by $v\mapsto J\circ dv -dv \circ J=2J\,
\overline\partial_Jv$, if $v\in\mathcal{C}^r(X,TX)$ is the 
infinitesimal variation of~$\psi$.

Let $\Gamma^r(X,Z,\mathcal{D})$ be the space of $\mathcal{C}^r$
embeddings of $X$ into $Z$ that are transverse to
$\mathcal{D}$. Transversality is an open condition, so
$\Gamma^r(X,Z,\mathcal{D})$ is an open subset in $\mathcal{C}^r(X,Z)$.
Now, $\Diff^r_0(X)$ acts on $\Gamma^r(X,Z,\mathcal{D})$ through the 
natural right action
$$\Gamma^r(X,Z,\mathcal{D})\times\Diff^r_0(X)\longrightarrow
\Gamma^r(X,Z,\mathcal{D}),\qquad(f,\psi)\mapsto f\circ\psi.$$

We wish to consider the differential of this action at point
$(f,\psi)$, $\psi=\Id_X$, with respect to the tangent space
isomorphisms of Lemma~\ref{banach-isom1}. This is just the addition
law in the bundle~$f^*TZ$:
$$
\mathcal{C}^r(X,f^*TZ)\times\mathcal{C}^r(X,TX)\to 
\mathcal{C}^{r-1}(X,f^*TZ),\qquad(u,v)\mapsto u+f_*v.
$$
A difficulty occurring here is the loss of regularity from $\mathcal{C}^r$
to $\mathcal{C}^{r-1}$ coming from the differentiations of $f$ and $v$. 
To overcome this difficulty, we have to introduce a slightly smaller
space of transverse embeddings.
%that will make possible to apply the implicit function theorem without trouble.

\begin{definition}For $r \in [1,+\infty] \cup \{\omega\}$ we consider the space
$$
\widetilde\Gamma^r(X,Z,\mathcal{D})\subset\Gamma^r(X,Z,\mathcal{D})
\subset \mathcal{C}^r(X,Z)
$$
of transverse embeddings $f:X\to Z$ such that $f$ is of class 
$\mathcal{C}^r$ as well as all ``transverse'' derivatives $h\cdot df$,
where $h$ runs over conormal holomorphic $1$-form with values in 
$(TZ/\mathcal{D})^*$. When $r=\infty$ or $r=\omega$ $($real analytic 
case$)$, we put $\smash{\widetilde\Gamma^r}(X,Z,\mathcal{D})=
\Gamma^r(X,Z,\mathcal{D})$.
\end{definition}
Then $\widetilde \Gamma^r(X,Z,\mathcal{D})$ satisfies the following conditions:
\begin{proposition}\label{prop-gamma-tilde-banach-1}
For $1 \leq r <\infty$, we have
\begin{itemize}
\item[\rlap{\kern-1cm\rm(i)}] The group $\Diff_0^{r+1}(X)$ acts on the right on 
$\widetilde\Gamma^r(X,Z,\mathcal{D})\,;$
\vskip0.2cm
\item[\rlap{\kern-1cm\rm(ii)}] The space $\widetilde\Gamma^r(X,Z,\mathcal{D})$ is a Banach manifold whose tangent space at a point \hbox{$f:X\to Z$} is
$
\mathcal{C}^r(X,f^*\mathcal{D})\oplus \mathcal{C}^{r+1}(X,TX).
$
\end{itemize}
\end{proposition}
\noindent
\begin{proof}
Part (i) is clear since $\Diff^{r+1}_0(X)$ acts on $\widetilde \Gamma^r(X,Z,\mathcal{D})$ through the natural right action
$$\widetilde \Gamma^r(X,Z,\mathcal{D})\times\Diff^{r+1}_0(X)\longrightarrow
\widetilde \Gamma^r(X,Z,\mathcal{D}),\qquad(f,\psi)\mapsto f\circ\psi.$$

For Part (ii), pick $f\in\smash{\widetilde\Gamma^r}(X,Z,\mathcal{D})$, 
$u\in \mathcal{C}^r(X,f^*\mathcal{D})$ and  $v\in\mathcal{C}^{r+1}(X,TX)$.
The flow of $v$ yields a family of diffeomorphisms $\psi_t\in\Diff_0^{r+1}(X)$
with $\psi_0=\Id_X$ and $\dot\psi_{t|t=0}=v$ (in the
sequel, all derivatives $\smash{\frac{d}{dt}_{|t=0}}$ will be indicated 
by a dot). Now, fix $\tilde u\in
\mathcal{C}^r(Z,\mathcal{D})$ such that $u=\tilde u\circ f$, by extending
the $\mathcal{C}^r$ vector field $f_*u$ from $f(X)$ to $Z$. The extension 
mapping $u\mapsto\tilde u$ can be chosen to be a continuous linear map
of Banach spaces, using e.g.\ a retraction from a tubular neighborhood
of the $\mathcal{C}^r$ submanifold $f(X)\subset Z$. Let $f_t$ be the flow
of $\tilde u$ starting at $f_0=f$ i.e.\ such that $\frac{d}{dt}f_t
=\tilde u(f_t)$.
Let $(e_j)_{1\le j\le N}$  be a local holomorphic frame of $TZ$ such
that $(e_j)_{n+1\le j\le N}$ is a holomorphic frame of $\mathcal{D}$,
$(e_j^*)$ its dual frame and $\nabla$ the unique local holomorphic 
connection of $TZ$ such that $\nabla e_j=0$. For $j=1,\ldots,n$, we find
$$
\frac{d}{dt}(e_j^*\circ df_t)
=e_j^*(f_t)\circ \nabla\frac{df_t}{dt}
=e_j^*(f_t)\circ \nabla(\tilde u(f_t))
=e_j^*(f_t)\circ (\nabla\tilde u)(f_t)\cdot df_t.
$$
However, if we write $\tilde u=\sum_{n+1\le k\le N}\tilde u_ke_k$ we see
that the composition vanishes since
$e_j^*e_k=0$. Therefore $\frac{d}{dt}(e_j^*\circ df_t)=0$ and
$e_j^*\circ df_t=e_j^*(f)\circ df\in\mathcal{C}^r(X)$. This shows that
$f_t\in \smash{\widetilde\Gamma^r}(X,Z,\mathcal{D})$ for all $t$, and
by definition we have $\dot f_t=\tilde u\circ f=u$. Now, if we define
$g_t=f_t\circ\psi_t$, we find $g_t\in \smash{\widetilde\Gamma^r}
(X,Z,\mathcal{D})$ by (i), and $\dot g_t=u+f_*v$ since $\dot\psi_t=v$.
The mapping $(u,v)\mapsto g_1=(f_t\circ\psi_t)_{|t=1}$ defines a local
``linearization'' of $\smash{\widetilde\Gamma^r}(X,Z,\mathcal{D})$ near $f$.
\end{proof}

We may consider now the differential of this action at point $(f,\psi)$, $f \in \widetilde \Gamma^r(X,Z,\mathcal{D})$ and $\psi=\Id_X$.
If we restrict $u$ to be in $\mathcal{C}^r(X,f^*\mathcal{D})$, we actually
get an isomorphism of Banach spaces
\begin{equation}\label{banach-isom3}
\mathcal{C}^r(X,f^*\mathcal{D})\times\mathcal{C}^r(X,TX)\to 
\mathcal{C}^r(X,f^*TZ),\qquad(u,v)\mapsto u+f_*v
\end{equation}
by the transversality condition. In fact, we can (non canonically)
define on $\widetilde \Gamma^r(X,Z,\mathcal{D})$ a ``lifting''
$$
\Phi(f,\bu):\mathcal{C}^r(X,f^*\mathcal{D})\to \tilde\Gamma^r(X,Z,\mathcal{D}),
\qquad u\mapsto \Phi(f,u)
$$
on a small neighborhood of the zero section, and the differential of
$\Phi(f,\bu)$ at $0$ is given by the inclusion
$\mathcal{C}^r(X,f^*\mathcal{D})\hookrightarrow
\mathcal{C}^r(X,f^*TZ)$. Modulo composition by
elements of $\Diff^{r+1}_0(X)$ close to identity (i.e.\ in the quotient
space $\widetilde \Gamma^r(X,Z,\mathcal{D})/\Diff^{r+1}_0(X)$), small deformations of
$f$ are parametrized by $\Phi(f,u)$ where $u$ is a small section of
$\mathcal{C}^r(X,f^*\mathcal{D})$. The first variation of $f$ depends
only on the differential of $\Phi$ along the zero section of $TZ$, so
it is actually independent of the choice of our map~$\Phi$. We can
think of small variations of $f$ as $f+u$, at least if we are working
in local coordinates $(z_1,\ldots,z_N)\in\mathbb{C}^N$ on~$Z$, and
consider that $\mathcal{D}_z\subset T_zZ=\mathbb{C}^N\,$; the use of a
map $\Phi$ like those already considered is however needed to make the
arguments global. Let us summarize these observations as follows.

\begin{lemma}\label{banach-isom2}For $ 1 \leq r<+\infty$, the quotient space 
$\widetilde \Gamma^r(X,Z,\mathcal{D})/\Diff^{r+1}_0(X)$ is a Banach manifold whose tangent
space at $f$ can be identified with $\mathcal{C}^r(X,f^*\mathcal{D})$
via the differential of the composition
$$\mathcal{C}^r(X,f^*\mathcal{D})~~
\mathop{\longrightarrow}^{\Phi(f,\bu)}~~\tilde\Gamma^r(X,Z,\mathcal{D})
\longrightarrow
\tilde\Gamma^r(X,Z,\mathcal{D})/\Diff^{r+1}_0(X)$$
at $0$, where the first arrow is given by $u\mapsto\Phi(f,u)$ and the second
arrow is the natural map to the quotient.\qed
\end{lemma}

Our next goal is to compute $J_f$ and the differential $dJ_f$ of
$f\mapsto J_f$ when $f$ varies in the above Banach manifold
$\widetilde \Gamma^r(X,Z,\mathcal{D})$.  Near a point $z_0\in Z$ we
can pick holomorphic coordinates $z=(z_1,\dots,z_N)$ 
centered at $z_0$, such that
$\mathcal{D}_{z_0}=\Span(\partial / \partial z_j)_{n+1 \leq j \leq N}$. Then
we have
\begin{equation}\label{D-generators}
\mathcal{D}_z=\Span\left(\frac{\partial}{\partial z_j}+
\sum_{1 \leq i \leq n}a_{ij}(z)
\frac{\partial}{\partial z_i}\right)_{n+1 \leq j \leq N},\ a_{ij}(z_0)=0.
\end{equation}
In other words $\mathcal{D}_z$ is the set of vectors of the form
$(a(z)\eta,\eta) \in \mathbb C^n \times \mathbb C^{N-n}$, where
$a(z)=(a_{ij}(z))$ is a holomorphic map into the space 
$\mathcal{L}(\mathbb{C}^{N-n},\mathbb{C}^n)$ of
$n\times(N-n)$ matrices. A trivial calculation shows that the vector fields
$e_j(z)=\frac{\partial}{\partial z_j}+\sum_ia_{ij}(z)
\frac{\partial}{\partial z_i}$ have brackets equal to
$$
[e_j,e_k]=\sum_{1 \leq i \leq n}\left(\frac{\partial a_{ik}}
{\partial z_j}(z_0) - \frac{\partial a_{ij}}{\partial z_k}(z_0)\right)
\frac{\partial}{\partial z_i}\ {\rm at}\ z_0,\ n+1 \leq j,k \leq N,
$$
in other words the torsion tensor $\theta$ is given by
\begin{equation}\label{torsion-expr}
\theta(z_0)=\sum_{1 \leq i \leq n,\;n+1 \leq j,k \leq N}
\theta_{ijk}(z_0)\,dz_j\wedge dz_k\otimes\frac{\partial}{\partial z_i},\quad
\theta_{ijk}(z_0)=\frac{1}{2}\left(\frac{\partial a_{ik}}
{\partial z_j}(z_0) - \frac{\partial a_{ij}}{\partial z_k}(z_0)\right).
\end{equation}
We now take a point $x_0\in X$ and apply this to $z_0=f(x_0)\in M=f(X) \subset Z$. 
With respect to coordinates $z=(z_1,\ldots,z_N)$ chosen as above,
we have $T_{z_0}M\oplus\Span(\partial/\partial z_j)_{n+1\le j\le N}=T_{z_0}Z$,
thus we can represent $M$ 
in the coordinates $z=(z',z'') \in \mathbb C^n \times 
\mathbb C^{N-n}$ locally as a graph $z''=g(z')$ in a small polydisc
$\Omega'\times\Omega''$ centered at $z_0$, and use 
$z'=(z_1,\dots,z_n)\in\Omega'$ as local (non holomorphic$\,$!) 
coordinates on~$M$. Here $g:\Omega'\to\Omega''$ 
is $\mathcal{C}^{r+1}$ differentiable and $g(z'_0)=z''_0$. 
The embedding $f:X\to Z$ is itself obtained
as the composition with a certain local $\mathcal C^{r}$ diffeomorphism  
$\varphi:X\supset V\to \Omega'\subset\mathbb{C}^n$, i.e.
$$
f=F\circ\varphi~~\hbox{on}~~V,
\quad\varphi:V\ni x\mapsto z'=\varphi(x)\in\Omega'\subset\mathbb{C}^n,\quad
F:\Omega'\ni z'\mapsto (z',g(z'))\in Z.
$$
With respect to the $(z',z'')$ coordinates, we get a
$\mathbb{R}$-linear isomorphism
$$
\begin{array}{lllll}
dF(z') &:&\mathbb{C}^n &\longrightarrow & T_{F(z')}M\subset T_{F(z')}Z
\simeq\mathbb{C}^n\times\mathbb{C}^{N-n}\\
\noalign{\vskip4pt}
 & &\kern3pt
\zeta & \longmapsto & (\zeta,dg(z')\cdot\zeta)=(\zeta,\partial g(z')
\cdot \zeta + \overline{\partial}g(z') \cdot \zeta).
\end{array} 
$$
Here $\overline{\partial}g$ is defined with respect to the standard
complex structure of $\mathbb C^n \ni z'$ and has a priori no
intrinsic meaning. The almost complex structure $J_f$ can be
explicitly defined by
\begin{equation}\label{def-Jf}
J_f(x)=d\varphi(x)^{-1}\circ J_{F}(\varphi(x))\circ d\varphi(x),
\end{equation}
where $J_F$ is the almost complex structure on $M$ defined by the embedding
$F:M\subset Z$, expressed in coordinates as $z'\mapsto (z',g(z'))$. We get
by construction
\begin{equation}\label{def-JF}
J_F(z')=dF(z')^{-1} \circ\pi_{Z,\mathcal{D},M}(F(z'))\circ J_Z(F(z')) \circ dF(z')
\end{equation}
where $J_Z$ is the complex structure on $Z$ and
$\pi_{Z,\mathcal{D},M}(z): T_zZ\rightarrow T_zM$ is the $\mathbb R$-linear 
projection to $T_zM$ along $\mathcal{D}_z$ at a point~$z\in M$. Since
these formulas depend on the first
derivatives of~$F$, we see that $J_f$ is at least of class $\mathcal{C}^{r-1}$ on $X$ and $J_F$ is at least of class
$\mathcal{C}^{r-1}$ on~$M$. We will see in Proposition~\ref{diff-cplx-struct-expr} that $J_f$ is in fact of class $\mathcal{C}^r$ on $X$ for $f \in \widetilde{\Gamma}^r(X,Z,\mathcal{D})$. Using the identifications
$T_{F(z')}M\simeq \mathbb C^n$, $T_zZ \simeq \mathbb C^N$ given by the
above choice of coordinates, we simply have $J_Z\eta = i\eta$ on $TZ$
since the $(z_j)$ are holomorphic, and we get therefore
$$
\begin{array}{lllll}
J_Z(F(z')) \circ dF(z') \cdot \zeta & = & idF(z') \cdot \zeta & = & i(\zeta,dg(z')\cdot \zeta) = (i\zeta,\partial g(z') \cdot i\zeta-\overline{\partial}g(z')\cdot i\zeta)\\
& & & = & (i\zeta,dg(z') \cdot i\zeta)-2(0,\overline{\partial}g(z') \cdot i\zeta).
\end{array}
$$
By definition of $z\mapsto a(z)$, we have
$(a(z)\eta,\eta)\in\mathcal{D}_z$ for every $\eta\in\mathbb{C}^{N-n}$, and so
$$
\pi_{Z,\mathcal{D},M}(z)(0,\eta)
=\pi_{Z,\mathcal{D},M}(z)\big((0,\eta)-(a(z)\eta,\eta)\big)
=-\pi_{Z,\mathcal{D},M}(z)(a(z)\eta,0).
$$
We take here $\eta=\overline{\partial}g(z') \cdot i\zeta$.
As $(i\zeta,dg(z') \cdot i\zeta)\in T_{F(z')}M$ already, we find
$$
\pi_{Z,\mathcal{D},M}(F(z'))\circ J_Z(F(z')) \circ dF(z') \cdot \zeta=
(i\zeta,dg(z') \cdot i\zeta)+2\pi_{Z,\mathcal{D},M}(F(z'))
\big(a(F(z'))\overline{\partial}g(z') \cdot i\zeta,0\big).
$$
From (\ref{def-JF}), we get in this way
\begin{equation}\label{formula-JF}
J_F(z') \cdot \zeta = i\zeta -2dF(z')^{-1} \circ\pi_{Z,\mathcal{D},M}(F(z'))\big(ia(F(z'))\overline{\partial}g(z') \cdot \zeta,0\big).
\end{equation}
In particular, since $a(z_0)=0$, we simply have $J_F(z'_0)\cdot \zeta = i\zeta$.
\vskip1mm

We want to evaluate the variation of the almost complex structure $J_f$ when
the embedding $f_t=F_t\circ\varphi_t$ varies with respect to some 
parameter $t\in[0,1]$. Let $w\in\mathcal{C}^r(X,f^*TZ)$
be a given infinitesimal variation of $f_t$ and $w=u+f_*v$, $u\in\mathcal{C}^r
(X,f^*\mathcal{D})$, $v\in\mathcal{C}^{r+1}(X,TX)$ its direct sum decomposition.
With respect to the trivia\-lization of $\mathcal{D}$ given by our
local holomorphic frame $(e_j(z))$, we can write in local coordinates
$$
u(\varphi^{-1}(z'))=\big(a(F(z'))\cdot\eta(z'),\eta(z')\big)\in\mathcal{D}_{F(z')}
$$
for some section $z'\mapsto\eta(z')\in\mathbb{C}^{N-n}$. Therefore
$$
u(\varphi^{-1}(z'))=\big(0,\eta(z')-dg(z')\cdot a(F(z'))\cdot\eta(z')\big)+
F_*\big(a(F(z'))\cdot\eta(z')\big)
$$
where the first term is ``vertical'' and the second one belongs to 
$T_{F(z')}M$. We then get a slightly different decomposition
$\widetilde w:=w\circ\varphi^{-1}=\widetilde u+ F_*\widetilde v
\in\mathcal{C}^r(\Omega',F^*TZ)$ where
\begin{eqnarray*}
&&\widetilde u(z')=
\big(0,\eta(z')-dg(z')\cdot a(F(z'))\cdot\eta(z')\big)\in
\{0\}\times\mathbb{C}^{N-n},\\
\noalign{\vskip4pt}
&&\kern1pt\widetilde v(z')=
\varphi_*v(z')+a(F(z'))\cdot\eta(z')\in\mathbb{C}^n.
\end{eqnarray*}
This allows us to perturb $f=F\circ\varphi$ as $f_t=F_t\circ\varphi_t$
with
\begin{equation}\label{embedding-variation}
\left\{\kern-15pt
\begin{matrix}
&&X\ni x\kern6pt\longmapsto z'=\varphi_t(x)
=\varphi(x)+t\widetilde v(\varphi(x))\in\mathbb{C}^n,\hfill\\
\noalign{\vskip6pt}
&&\mathbb{C}^n\ni z'\longmapsto F_t(z')=(z',g_t(z'))\in Z,\hfill\\
\noalign{\vskip5pt}
&&g_t(z')=g(z')+t\widetilde u(z')
=g(z')+t\big(\eta(z')-dg(z')\cdot a(F(z'))\cdot\eta(z')\big),\hfill
\end{matrix}
\right.
\end{equation}
in such a way that $\dot f_t=\smash{\frac{d}{dt}(f_t)_{|t=0}}=w$.
%\,$; in the , all derivatives $\smash{\frac{d}{dt}_{|t=0}}$ will be indicated by a dot.
We replace 
$f,\,g,\,F,\,M$ by $f_t,\,g_t,\,F_t,\,M_t$ in (\ref{formula-JF}) and compute
the derivative for $t=0$ and $z'=z'_0$. Since $a(z_0)=0$, the only non zero 
term is the one involving the derivative of the map~$t \mapsto a(F_t(z'))$. 
We have $\smash{\dot F_t}(z'_0)=(0,\eta(z'_0))=u(x_0)$ where 
$\eta(z'_0)\in\mathbb{C}^{N-n}$, thus $\smash{\dot J_{F_t}}$ can be
expressed at $z'_0$ as
$$
\dot J_{F_t}(z'_0)\cdot\zeta:=
\frac{d}{dt}\big(J_{F_t}(z'_0)\cdot\zeta\big)_{|t=0}=
-2dF(z'_0)^{-1}\circ\pi_{Z,\mathcal{D},M}(z'_0)\big(
ida(z_0)(u(x_0))\cdot\overline\partial g(z'_0)\cdot\zeta,0\big).
$$
Now, if we put $\lambda=ida(z_0)(u(x_0))\overline{\partial}g(z'_0)
\cdot\zeta$, as $\mathcal{D}_{z_0}=\{0\} \times \mathbb C^{N-n}$
in our coordinates, we immediately get
$$
\pi_{Z,\mathcal{D},M}(z'_0)(\lambda,0) = (\lambda,dg(z'_0)\cdot\lambda)
=dF(z'_0)\cdot\lambda ~~\Longrightarrow~~
dF(z'_0)^{-1} \circ \pi_{Z,\mathcal{D},M}(z'_0)(\lambda,0) = \lambda.
$$
Therefore, we obtain the very simple expression
\begin{equation}\label{dot-JF}
\dot J_{F_t}(z'_0)=
-2i\,da(z_0)(u(x_0))\cdot \overline{\partial}g(z'_0)\in
\End_{\overline{\mathbb{C}}}(\mathbb{C}^n)
\end{equation}
where $da(z_0)(\xi)\in\mathcal{L}(\mathbb{C}^{N-n},\mathbb{C}^n)$ is the derivative
of the matrix function $z\mapsto a(z)$ at point $z=z_0$ in the direction $\xi\in\mathbb{C}^N$, and $\overline\partial g(z'_0)$ is viewed as an element of
$\mathcal{L}_{\overline{\mathbb{C}}}(\mathbb{C}^n,\mathbb{C}^{N-n})$.
What we want is the derivative of 
$J_{f_t}=d\varphi_t^{-1}\circ J_{F_t}(\varphi_t)\circ d\varphi_t$ at $x_0$
for $t=0$. Writing $\varphi_*$ as an abbreviation for $d\varphi$, we find
for $t=0$
\begin{equation}\label{dot-Jf}\kern-15pt
\begin{matrix}
&&\dot J_{f_t} = 
-\varphi_*^{-1}\circ d\dot\varphi_t\circ\varphi_*^{-1}\circ J_F(\varphi)
\circ\varphi_*+\varphi_* ^{-1}\circ J_F(\varphi)\circ d\dot\varphi_t
+\varphi_*^{-1}\circ\dot J_{F_t}(\varphi) \circ\varphi_*\hfill\\
\noalign{\vskip5pt}
&&\phantom{\dot J_{f_t}} = 
2J_f\,\overline\partial_{J_f}(\varphi_* ^{-1}\dot\varphi_t)
+\varphi_*^{-1}\circ\dot J_{F_t}(\varphi) \circ\varphi_*,\hfill
\end{matrix}
\end{equation}
where the first term in the right hand side comes from the identity 
$-ds\circ J_f+J_f\circ ds=2J_f\,\overline\partial_{J_f}s$
with $s=\varphi_*^{-1}\dot\varphi_t\in\mathcal{C}^r(X,TX)$ and 
$ds=\varphi_*^{-1}d\dot\varphi_t$. Our choices
$\widetilde v=\varphi_*v+a\circ F\cdot \eta$ and 
$\varphi_t=\varphi+t\widetilde v\circ\varphi$ yield
$$
\dot\varphi_t=\widetilde v\circ\varphi=
\varphi_*v+a\circ f\cdot\eta\circ\varphi~~\Longrightarrow~~
\varphi_*^{-1}\dot\varphi_t=v+\varphi_*^{-1}(a\circ f\cdot \eta\circ\varphi).
$$
If we recall that $a(z_0)=0$ and $\eta(\varphi(x_0))=\eta(z'_0)=\pr_2 u(x_0)$,
we get at $x=x_0$
\begin{equation}\label{dbar-tilde-v}
\overline\partial_{J_f}(\varphi_*^{-1}\dot\varphi_t)(x_0)=
\overline\partial_{J_f}v(x_0)+\varphi_*^{-1}\big(
da(z_0)(\overline\partial_{J_f}f(x_0))\cdot\pr_2 u(x_0)\big).
\end{equation}
By construction, $\varphi_*=d\varphi$ is compatible with
the respective almost complex structures
$(X,J_f)$ and $(\mathbb{C}^n,J_F)$. A combination
of (\ref{dot-JF}), (\ref{dot-Jf}) and (\ref{dbar-tilde-v}) yields
$$
\dot J_{f_t}(x_0) = 2J_f\,\overline\partial_{J_f}v(x_0)+
\varphi_*^{-1}\Big(
2i\,da(z_0)(\overline\partial_{J_f}f(x_0))\cdot \pr_2u(x_0)
-2i\,da(z_0)(u(x_0)) \cdot \overline\partial g(z'_0)\circ\varphi_*\Big).
$$
As $\overline\partial_{J_f}f(x_0)= (\overline\partial_{J_F}F)(z'_0)\circ
d\varphi(x_0)=(0,\overline\partial g(z'_0))\circ\varphi_*$ and 
$f_*=F_*\circ\varphi_*$, we get
$$
\dot J_{f_t}(x_0) = 
f_*^{-1}F_*\Big(2i\,da(z_0)(\overline\partial_{J_f}f(x_0))\cdot\pr_2u(x_0)
-2i\,da(z_0)(u(x_0)) \cdot \pr_2\overline\partial_{J_f}f(x_0)\!\Big)
+2J_f\,\overline\partial_{J_f}v(x_0).
$$
By (\ref{torsion-expr}), the torsion tensor $\theta(z_0):
\mathcal{D}_{z_0}\times\mathcal{D}_{z_0}\to T_{z_0}Z/\mathcal{D}_{z_0}
\simeq F_*T_{z_0}M=f_*T_{x_0}X$ is given by
$$
\theta(\eta,\lambda)=\sum_{1\le i\le n,\,n+1\le j,k\le N}
\left(\frac{\partial a_{ik}}
{\partial z_j}(z_0) - \frac{\partial a_{ij}}{\partial z_k}(z_0)\right)
\eta_j\lambda_k\frac{\partial}{\partial z_i}=
da(z_0)(\eta)\cdot\lambda-da(z_0)(\lambda)\cdot\eta.
$$
Since our point $x_0\in X$ was arbitrary and $\dot J_{f_t}(x_0)$ is the value
of the differential $dJ_f(w)$ at~$x_0$, we finally get the global formula
$$
dJ_f(w)=2J_f\big(f_*^{-1}\theta(\overline\partial_{J_f}f,u)
+\overline\partial_{J_f}v\big)
$$
(observe that $\overline\partial_{J_f}f\in{\mathcal{L}}_{\overline{\mathbb{C}}}
(TX,f^*TZ)$ actually takes values in $f^*\mathcal{D}$, so taking a projection
to $f^*\mathcal{D}$ is not needed). We conclude$\,$:

\begin{proposition}\label{diff-cplx-struct-expr} Let $r \in [1,+\infty] \cup \{\omega\}$.
\begin{itemize}
 \item[\rlap{\kern-1cm\rm(i)}] The natural map $f\mapsto J_f$ sends 
$\widetilde\Gamma^r(X,Z,\mathcal{D})$ into $\mathcal{J}^r(X).$
\item[\rlap{\kern-1cm\rm(ii)}]
The differential of the natural map
$$\widetilde \Gamma^r(X,Z,\mathcal{D})\to \mathcal{J}^{r}(X),\qquad f\mapsto J_f$$
along every infinitesimal variation $w=u+f_*v:X\to f^*TZ=
f^*\mathcal {D}\oplus f_*TX$ of $f$ is given by
$$
dJ_f(w)=2J_f\big(f_*^{-1}\theta(\overline\partial_{J_f}f,u)
+\overline\partial_{J_f}v\big)
$$
where $\theta:\mathcal{D}\times\mathcal{D}\to TZ/\mathcal{D}$
is the torsion tensor of the holomorphic distribution $\mathcal{D}$,
and $\overline\partial f=\overline\partial_{J_f} f$,
$\overline\partial v=\overline\partial_{J_f}v$
are computed with respect to the almost complex structure $(X,J_f)$.
\item[\rlap{\kern-1cm\rm(iii)}] The differential $dJ_f$ of $f\mapsto J_f$
on $\widetilde\Gamma^r(X,Z,\mathcal{D})$ is a continuous morphism
$$
\mathcal{C}^r(X,f^*\mathcal{D})\oplus \mathcal{C}^{r+1}(X,TX)\longrightarrow
\mathcal{C}^r(X,\End_{\overline{\mathbb{C}}}(TX)),\quad
(u,v)\longmapsto 2i\big(\theta(\overline\partial f,u)+\overline\partial v\big).
$$
\end{itemize}
\end{proposition}

\noindent
If $r= +\infty$ or $r=\omega$ then we replace $r+1$ by $r$ in Condition (iii).

\begin{proof}
Parts (i) and (ii) are clear, as it can be easily seen that $\overline\partial f$ depends only on the transversal part of $df$ by the very definition of $J_f$ and of $\overline\partial f=\frac{1}{2}(df+J_Z\circ df\circ J_f)$.

Part (iii) is a trivial consequence of the general variation formula.
\end{proof}

Our goal, now, is to understand under which conditions $f\mapsto J_f$ can
be a local submersion from $\smash{\widetilde\Gamma^r}(X,Z,\mathcal{D})$
to  $\mathcal{J}^r(X)$. If we do not take into account the quotient by
the action of $\Diff_0^{r+1}$ on  $\mathcal{J}^r(X)$, we obtain a more
demanding condition. For that stronger requirement, we see that a 
sufficient condition is that the continuous linear map
\begin{equation}\label{partial-d-Jf}
\mathcal{C}^r(X,f^*\mathcal{D})\longrightarrow
\mathcal{C}^r(X,\End_{\overline{\mathbb{C}}}(TX)),\quad
u\longmapsto 2i\,\theta(\overline\partial f,u)
\end{equation}
be surjective.

\begin{theorem}\label{theorem-var-Jf} Fix $r\in[1,\infty]\cup\{\omega\}$
$($again, $\omega$ means real analyticity here$)$. Let $(Z,\mathcal{D})$ be 
a complex manifold equipped with a holomorphic distribution, and let $f\in
\smash{\widetilde\Gamma^r}(X,Z,\mathcal{D})$ be a transverse embedding
with respect to $\mathcal{D}$. Assume that $f$ and the torsion 
tensor $\theta$ of $\mathcal{D}$ satisfy the following additional 
conditions$\,:$
\vskip1mm
\begin{itemize}[leftmargin=2\parindent]
\item[{\rm(i)}] $f$ is a totally real embedding, i.e.\ 
$\overline\partial f(x)\in\End_{\overline{\mathbb{C}}}(T_xX,T_{f(x)}Z)$ is 
injective at every point $x\in X\,;$
\vskip1mm
\item[{\rm(ii)}] for every $x\in X$ and every 
$\eta\in \End_{\overline{\mathbb{C}}}(TX)$, there exists a vector
$\lambda\in\mathcal{D}_{f(x)}$ such that $\theta(\overline\partial f(x)\cdot
\xi,\lambda)=\eta(\xi)$ for all $\xi\in TX$.
\end{itemize}
\vskip1mm
Then there is a neighborhood $\mathcal{U}$ of $f$ in $\smash{\widetilde\Gamma^r}(X,Z,\mathcal{D})$ and a neighborhood $\mathcal {V}$ of $J_f$ in 
$\mathcal{J}^r(X)$ such that $\mathcal{U}\to\mathcal{V}$,
$f\mapsto J_f$ is a submersion.
\end{theorem}

\begin{proof} This is an easy consequence of the implicit function theorem in the Banach space situation $r<+\infty$. Let $\Phi$ be the real analytic map \hbox{$TZ\to Z\times Z$} considered in section~\ref{deform-section}, and let 
$$
\Psi_f:\mathcal{C}^r(X,f^*\mathcal{D)}\longrightarrow
\widetilde\Gamma^r(X,Z,\mathcal{D}),\qquad u\mapsto\Phi(f,f_*u).
$$
By definition $f=\Psi_f(0)$ and $\Psi_f$ defines the infinite dimensional
manifold structure on $\smash{\widetilde\Gamma}^r(X,Z,\mathcal{D})$ by
identifying a neigborhood of $0$ in the topological vector space 
$\mathcal{C}^r(X,f^*\mathcal{D)}$ with
a neighborhood of $f$ in~$\smash{\widetilde\Gamma}^r(X,Z,\mathcal{D})$, and 
providing in  this way a ``coordinate chart''.
As we have seen in (\ref{partial-d-Jf}), the differential $u\mapsto dJ_f(u)$ 
is given by 
$$
u\mapsto L_f(u)=2i\,\theta(\overline\partial f,u)
$$ where $L_f\in\mathcal{C}^r(X,\Hom(f^*\mathcal{D},
\End_{\overline{\mathbb{C}}}(TX)))$ is by our assumption (ii) a surjective 
morphism of bundles of
finite rank. The kernel $\mathcal{K}:=\Ker L_f$ is a 
$\mathcal{C}^r$ subbundle of $f^*\mathcal{D}$, thus we can 
select a $\mathcal{C}^r$ subbundle 
$\mathcal{E}$ of $f^*\mathcal{D}$ such that
$$
f^*\mathcal{D}=\mathcal{K}\oplus \mathcal{E}.
$$
(This can be seen by a partition of unity argument for $r\ne\omega\,$; in 
the real analytic case, one can instead complexify the real analytic objects
and apply a Steinness argument together with Cartan's theorem B to obtain 
a splitting).  The~differential of the composition 
$$u\mapsto g=\Psi_f(u),\qquad g\mapsto J_g$$
is precisely the restriction of $L_f=dJ_f$ to sections
$u\in\mathcal{C}^r(X,\mathcal{E})\subset\mathcal{C}^r(X,f^*\mathcal{D})$, 
which is by construction a bundle isomorphism from 
$\mathcal{C}^r(X,\mathcal{E})$ 
onto~$\mathcal{C}^r(X,\End_{\overline{\mathbb C}}(TX))$. Hence
for $r<\infty$, $u\mapsto g=\Psi_f(u)\mapsto J_{\Psi_f(u)}$ is a 
$\mathcal{C}^r$-diffeomomorphism 
from a neighborhood $\mathcal{W}^{\mathcal{E}}_r(0)$ of the zero section of 
$\mathcal{C}^r(X,\mathcal{E})$ 
onto a neighborhood $\mathcal{V}_r$ of $J_f\in\mathcal{J}^r(X)$, and so
$g\mapsto J_g$ is a $\mathcal{C}^r$-diffeomomorphism from
$\mathcal{U}^{\mathcal{E}}_r:=\Psi(\mathcal{W}^{\mathcal{E}}_r(0))$ 
onto $\mathcal{V}_r$. This
argument does not quite work for $r=\infty$ or $r=\omega$, since we do not
have Banach spaces. Nevertheless, for $r=\infty$, we can apply the result
for a given finite $r_0$ and consider $r'\in[r_0,\infty[$ arbitrarily large. 
Then, by applying a local diffeomorphism argument in $\mathcal{C}^{r'}$
at all nearby points $g=\Psi_f(u)$ (and by using the injectivity on
$\mathcal{U}^{\mathcal{E}}_{r_0}$), we see that the map
$$\mathcal{U}^{\mathcal{E}}_{r'}:=\Psi_f(\mathcal{W}^{\mathcal{E}}_{r_0}(0)\cap\mathcal{C}^{r'}(X,\mathcal{E}))\longrightarrow
\mathcal{V}_{r_0}\cap\mathcal{J}^{r'}(X),\qquad
g\mapsto J_g$$
is a $\mathcal{C}^{r'}$-diffeomorphism. Since this is true for all $r'$ with 
the ``same'' neighborhood, i.e.\ one given by the same semi-norms of order
$r_0$ and the same bounds, the case $r=\infty$ also yields
a local diffeomorphism of Fr\'echet manifolds. For 
$r=\omega$, we have instead an inductive limit of Banach spaces of real 
analytic sections $u\in\mathcal{C}^{\omega}_\rho(X,\mathcal{E})$ whose Taylor
expansions $u(y)=\sum u_\alpha(x)(y-x)^\alpha$ converge uniformly 
on tubular neighborhoods of the diagonal,
of shrinking radii~$\rho\to 0$ (with respect to a given real analytic 
atlas of~$X$, say). The argument is quite similar, by 
considering the intersection 
$\mathcal{W}_{r_0}(0)\cap\mathcal{C}^{\omega}_\rho(X,\mathcal{E})$
we get a diffeomorphism onto a neighborhood of $J_f$ in 
$\mathcal{J}^{\omega}_\rho(X)$, if we take $\rho$
smaller than the radius of convergence $\rho_0$ that can be used for
$f$, $\Psi_f$ and~$\mathcal{E}$. We still have to justify the fact that 
$g\mapsto J_g$ is a local submersion near~$f$. Again, 
for finite values of $r$, e.g.\ for~$r=r_0$, this is true by the
Banach case of the implicit function theorem.
Then the fibers $\{g\,;\;J_g=J\}$ of $g\mapsto J_g$ are Banach 
manifolds modelled on
$\mathcal{C}^{r_0}(X,\mathcal{K})$ in a suitable neighborhood
$$
\mathcal{U}_{r_0}:=\Psi_f(\mathcal{W}^{\mathcal{K}}_{r_0}\oplus
\mathcal{W}^{\mathcal{E}}_{r_0})\subset
\widetilde\Gamma^r(X,Z,\mathcal{D})~~\hbox{of $f$}
$$
where $\mathcal{U}_{r_0}$ is obtained as the image by $\Psi_f$ of a sufficiently small neighborhood of $0$ in
$$
\mathcal{C}^{r_0}(X,f^*\mathcal{D})=
\mathcal{C}^{r_0}(X,\mathcal{K})\oplus \mathcal{C}^{r_0}(X,\mathcal{E}).
$$
Observe that the ``central'' fiber $\{g\,;\;J_g=J_f\}$ is in fact tangent to
$\mathcal{C}^{r_0}(X,\mathcal{K})\subset \mathcal{C}^{r_0}(X,f^*\mathcal{D})$, and
that by continuity, $\Ker dJ_g$ is a supplementary subspace of
$\mathcal{C}^{r_0}(X,\mathcal{E})$ for $g$ close to $f$ in 
$\mathcal{C}^{r_0}$ topology, $r_0\ge 1$.
We conclude by considering neighborhoods of $f\in
\smash{\widetilde\Gamma}^r(X,Z,\mathcal{D})$
$$
\mathcal{U}_r=\Psi_f\big((\mathcal{W}^{\mathcal{K}}_{r_0}\oplus
\mathcal{W}^{\mathcal{E}}_{r_0})\cap\mathcal{C}^r(X,f^*\mathcal{D})\big)
$$
that are ``uniform'' in~$r$.
\end{proof}
\begin{remark}\label{IFT-remark}{\rm
(a) When $\mathcal{D}$ is a foliation, i.e.\ $\theta\equiv 0$ 
identically, or
when $f$ is holomorphic or pseudo-holomorphic, i.e.\ 
$\overline\partial f=0$, we have
$dJ_f\equiv 0$ up to the action of $\Diff^{r+1}_0(X)$. Therefore, when $n>1$,
one can never attain the submersion property by means of a
foliation $\mathcal{D}$ or when starting from a (pseudo-)holomorphic map~$f$.

\vskip4pt\noindent
(b) Condition (ii) of Theorem.~\ref{theorem-var-Jf} is easily seen to be equivalent
to (\ref{partial-d-Jf}). When one of these is satisfied, condition (i)
on the injectivity of $\overline\partial f$ is in fact automatically
implied: otherwise a vector $\xi\in\Ker\overline\partial f(x)$ could never
be mapped to a nonzero element $\eta(\xi)$ assigned by~$\eta$.
\vskip4pt\noindent
(c) For condition (ii) or (\ref{partial-d-Jf}) to be satisfied, 
a necessary condition is that the rank $N-n$ of $\mathcal{D}$ be 
such that
$$
N-n\ge \rank(\End_{\overline{\mathbb{C}}}(TX))=n^2,
$$
i.e.\ $N\ge n^2+n$, so the dimension of $Z$ must be rather large compared
to $n=\dim_{\mathbb{C}} X$.
\vskip4pt\noindent
We will see in the next section that is indeed possible to find a
quasi-projective algebraic variety $Z$ whose dimension is quadratic 
in $n$, for which any $n$-dimensional almost complex manifold $(X,J)$
admits a transverse embedding $f:X\hookrightarrow Z$
satisfying (i), (ii) and $J=J_f$. The present remark shows that one
cannot improve the quadratic character $N=O(n^2)$ of the  embedding dimension
under condition (ii).}
\end{remark}

\section{Universal embedding spaces}\label{emb-sect}

We prove here the existence of the universal embedding spaces $(Z_{n,k},
\mathcal{D}_{n,k})$ claimed in Theorem~\ref{theorem-univ-embed}. They will 
be constructed as some sort of combination of Grassmannians and twistor bundles.
For $k>n$, we let $W\subset\mathbb{R}^{2k}\times G_{\mathbb{R}}(2k,2n)\times
\End_{\mathbb{R}}(\mathbb{R}^{2k})$ be the set of triples
$(w,S,J)$ where $w\in\mathbb{R}^{2k}$, $S$ lies in the real
Grassmannian of $2n$-codimensional subspaces of $\mathbb{R}^{2k}$,
$J\in\End_{\mathbb{R}}(\mathbb{R}^{2k})$ satisfies $J^2=-I$ and $J(S)\subset S$.
Clearly, $W$ is a quasi-projective real algebraic variety, and it has
a complexification $W^{\mathbb{C}}$ which can be described as a component
of the set of triples
$$
(z,S,J)\in\mathbb{C}^{2k}\times G_{\mathbb{C}}(2k,2n)\times
\End_{\mathbb{C}}(\mathbb{C}^{2k})
$$
such that $J^2=-I$ and $J(S)\subset S$. Such an endomorphism $J$ actually 
induces almost complex structures on $\mathbb{C}^{2k}$ and on $S$, and thus
yields direct sum decompositions $\mathbb{C}^{2k}=\Sigma'\oplus\Sigma''$ and
$S=S'\oplus S''$ where $S'\subset\Sigma'$, $S''\subset\Sigma''$ correspond
respectively to the $+i$ and $-i$ eigenspaces.
If $J$ is the complexification of some $J^{\mathbb{R}}\in 
\End_{\mathbb{R}}(\mathbb{R}^{2k})$ and $S$ is the complexification of 
some $S^{\mathbb{R}}\subset\mathbb{R}^{2k}$, we have 

\begin{equation}\label{dim-eq}
\dim S'=\dim S''=\frac{1}{2}\dim S=k-n\quad\hbox{and}\quad
\dim \Sigma'=\dim \Sigma''=\frac{1}{2}\dim\mathbb{C}^{2k}=k. 
\end{equation}

We let $Z$ be the irreducible nonsingular quasi-projective algebraic variety
consisting of triples $(z,S,J)$ as above where $J$ has such ``balanced'' 
eigenspaces $S'$, $S''$, $\Sigma'$, $\Sigma''$. Alternatively, we could 
view $Z$ as the set of $5$-tuples $(z,S',S'',\Sigma',\Sigma'')$ with
$S'\subset\Sigma'$, $S''\subset \Sigma''$ and $\mathbb{C}^{2k}=\Sigma'\oplus
\Sigma''$, and with dimensions given as above (the decomposition 
$\mathbb{C}^{2k}=\Sigma'\oplus\Sigma''$ then defines $J$ uniquely).
Therefore we have by (\ref{dim-eq})
$$
N:=\dim_{\mathbb{C}} Z=2k+2(k^2+n(k-n))
$$
since $k^2$ is the dimension of the Grassmannian of subspaces $\Sigma'\subset
\mathbb{C}^{2k}$ (or $\Sigma''\subset\mathbb{C}^{2k}$), and 
$n(k-n)$ the dimension of the Grassmannian of subspaces $S'\subset\Sigma'$
(or $S''\subset\Sigma''$). The real part $W=Z^{\mathbb{R}}\subset Z$ can also
be seen as the set of $5$-tuples $p=(w,S',S'',\Sigma',\Sigma'')$ for 
which $w=z=\overline z\in\mathbb{R}^{2k}$, $S''=\overline{S'}$ and 
$\Sigma''=\overline{\Sigma'}$. 

In our first interpretation, the tangent space $TZ$ at a point 
$p=(z,S,J)$ consists of triples $(\zeta,u,v)$ where $\zeta\in\mathbb{C}^{2k}$, 
$u\in \Hom(S,\mathbb{C}^{2k}/S)$ and $v\in \End(\mathbb{C}^{2k})$ is such that
\hbox{$v\circ J+J\circ v=0$} and $v(S)\subset S$. In the second 
interpretation, $T_pZ$ is given by $5$-tuples $(\zeta,u',u'',v',v'')$ with
$\zeta\in\mathbb{C}^{2k}$, 
$u'\in\Hom(S',\Sigma'/S')$,
$u''\in\Hom(S'',\Sigma''/S'')$,
$v'\in\Hom(\Sigma',\mathbb{C}^{2k}/\Sigma')$ and
\hbox{$v''\in\Hom(\Sigma'',\mathbb{C}^{2k}/\Sigma'')$}. [In order to check these relations, it may be useful to use coordinate charts, constructed e.g.\ by considering a fixed $J$-stable complementary subspace $S\oplus T=\mathbb{C}^{2k}$, and, points of the Grassmannian close to $S$ being then seen as graphs of maps $u\in\Hom(S,T)$ -- we leave these details to the reader]. We let $\mathcal{D}_p\subset T_pZ$ be the set of $5$-tuples 
$(\zeta,u',u'',v',v'')$ for which
$\zeta\in S'\oplus\Sigma''\subset \mathbb{C}^{2k}$ (with no conditions on the
other components $(u',u'',v',v'')$). Therefore we have a canonical isomorphism
$T_pZ/\mathcal{D}_p\simeq\Sigma'/S'$, and we see that $\mathcal{D}$ is an
algebraic subbundle of corank $n$, i.e.\ $\rank(\mathcal{D})=N-n$, and
$TZ/\mathcal{D}$ is isomorphic to the tautological bundle $\Sigma'/S'$ arising
from the flag manifold structure of pairs $(S',\Sigma')$ with
$S'\subset\Sigma'\subset\mathbb{C}^{2k}$.

\begin{proof}[Proof of Theorem \ref{theorem-univ-embed}]
Let $(X,J_X)$ be an arbitrary compact $n$-dimensional almost complex
manifold, where $J_X$ is of class $\mathcal{C}^{r+1}$, $r\in[0,\infty]
\cup\{\omega\}$. We may assume here that the differential structure
of $X$ itself is $\mathcal{C}^\omega$. 
Since $\dim_{\mathbb{R}}X=2n$, the strong Whitney embedding theorem
\cite{whitney} shows that there exists a $\mathcal{C}^\omega$ embedding 
$g:X\to\mathbb{R}^k$ where  $k=2(2n)=4n$. (By well-known results, one can
even take $X$ to be given by a real algebraic variety and $g$ to be algebraic, see \cite{tognoli}). Let $NX$ be the normal bundle
of $g(X)$ in $\mathbb{R}^k$ (with a slight abuse of notation consisting
of identifying $X$ and $g(X)$). We use here the Euclidean structure
of $\mathbb{R}^k$ to view $NX$ as a subbundle of the trivial tangent
bundle $T\mathbb{R}^k$. Next, we embed $X$ in $\mathbb{R}^{2k}$ by the
diagonal embedding $x\mapsto G(x)=(g(x),g(x))$, whose normal bundle is
$TX\oplus NX\oplus NX$. We have
$$
T\mathbb{R}^{2k}_{|G(X)}= TX\oplus TX\oplus NX\oplus NX.
$$
On $NX\oplus NX$ (or, for that purpose, on 
the double of any real vector bundle), there is a tautological almost 
complex structure $J_{NX\oplus NX}$ given by $(u,v)\mapsto (-v,u)$.
For every $x\in X$, we consider the complex structure $\widetilde J(x)$ on
$T\mathbb{R}^{2k}_{|G(x)}=\mathbb{R}^{2k}$ defined by
$$\widetilde J(x):=J_X(x)\oplus (-J_X(x))\oplus J_{NX\oplus NX}(x).$$
Notice that $(X,-J_X)$ is the complex conjugate almost complex manifold
$\overline X$. In some sense, we have embedded $X$ diagonally into
$X\times \overline X$ (this embedding is totally real and has normal
bundle $T\overline X$), and 
composed that diagonal embedding with the product embedding
$$
g\times g:X\times\overline X\to\mathbb{R}^k\times\mathbb{R}^k=
\mathbb{R}^{2k}
$$
which has normal bundle $\pr_1^*NX\oplus \pr_2^*NX$. Let
$\widetilde J^{\mathbb{C}}(x)\in\End(\mathbb{C}^{2k})$ be the complexification
of $\widetilde J(x)$, and let $\Sigma'_x\subset \mathbb{C}^{2k}$, 
$\Sigma''_x\subset\mathbb{C}^{2k}$ be the $+i$ and $-i$ eigenspaces
of $\widetilde J(x)$ respectively (both are $k$-dimensional). By
construction, the bundle $\Sigma'$ consists of vectors of the form
$(\xi^{1,0},\eta^{0,1},u,-iu)$, $\xi^{1,0}\in T^{1,0}X$,
$\eta^{0,1}\in T^{0,1}X$, $u\in N^{\mathbb{C}}X$, and similarly
$\Sigma''$ consists of vectors of the form $(\xi^{0,1},\eta^{1,0},u,iu)$.
We further define $S^{\mathbb{R}}\subset T\mathbb{R}^{2k}$ and its fiberwise
complexification $S_x=S_x^{\mathbb{R}}\otimes\mathbb{C}\subset\mathbb{C}^{2k}$ by
$$
S^{\mathbb{R}}=
\{0\}\oplus T\overline X\oplus NX\oplus NX,\qquad
S=\{0\}\oplus T^{\mathbb{C}}\overline X\oplus N^{\mathbb{C}}X\oplus 
N^{\mathbb{C}}X.
$$
Clearly $S^{\mathbb{R}}_x$ is stable by $\widetilde J(x)$ and
\begin{eqnarray*}
&&S'\,:=\,\Sigma'\cap S=\{0\}
\oplus T^{0,1}X\oplus\{(u,-iu),\;u\in N^{\mathbb{C}}X\},\\
&&S'':=\Sigma''\cap S=\{0\}
\oplus T^{1,0}X\oplus\{(u,iu),\;u\in N^{\mathbb{C}}X\}
\end{eqnarray*}
are the $+i$ and $-i$ eigenspaces of $\widetilde J^{\mathbb{C}}_{|S}$,
respectively. We finally get an embedding of class $\mathcal{C}^{r+1}$
$$
f:X\hookrightarrow Z,\qquad x\mapsto\big(G(x),S'_x,S''_x,
\Sigma'_x,\Sigma''_x\big),
$$
and since $(TZ/\mathcal{D})_{f(x)}\simeq \Sigma'_x/S'_x\simeq T^{1,0}_xX$,
we see that the almost complex structure $J_f$ induced by
the natural complex structure of $TZ/\mathcal{D}$ coincides 
with $J_X$. As this point, $Z$ is quasi-projective but not affine. 
However $f(X)$ is contained in the real part $W=Z^{\mathbb{R}}$, especially
the corresponding subspaces $S=S'\oplus S''$ lie in the real part 
$G_{\mathbb{R}}(2k,2n)\subset G(2k,2n)$ of the complex Grassmannian.
In this situation, we can find an ample divisor $\Delta$ of $G(2k,2n)$ 
that is disjoint from $G_{\mathbb{R}}(2k,2n)$
and invariant by complex conjugation (to see this, we embed the
Grassmannian into a complex projective space $\mathbb{CP}^s$ by the 
Pl\"ucker embedding, and observe that the real hyperquadric
$Q=\{\sum_{0\le j\le s} z_j^2=0\}$ is disjoint from $\mathbb{RP}^s$;
we can thus take $\Delta$ to be the inverse image of $Q$ by the 
Pl\"ucker embedding). By restricting the situation to
the complement $G(2k,2n)\smallsetminus\Delta$, we obtain an affine
algebraic open set $Z'\subset Z$ that is invariant by 
complex conjugation, so that $f(X)\subset Z^{\prime\,\mathbb{R}}$.
Theorem~\ref{theorem-univ-embed}  is proved with $Z_{n,k}=Z'$
and~$\mathcal{D}_{n,k}=\mathcal{D}_{|Z'}$.
\end{proof}

\begin{remark} {\rm A computation in coordinates shows that conditions
(i) and (ii) of Theorem~\ref{theorem-var-Jf} are satisfied in this construction.
Actually (i) is already implied by the fact that the image 
$M=f(X)\subset Z_{n,k}^{\mathbb{R}}$ is totally real.}
\end{remark}

\begin{remark} {\rm It is easy to find a non singular model for
a projective compactification $\overline Z_{n,k}$ of $Z_{n,k}$: just 
consider the set of $5$-tuples $p=(z,S',S'',\Sigma',\Sigma'')$ where
$z\in\mathbb{CP}^{2k}$, \hbox{$S'\subset\Sigma'\subset T\mathbb{CP}^{2k}$},
$S''\subset\Sigma''\subset T\mathbb{CP}^{2k}$, so that 
$\pi:\overline Z_{n,k}\to \mathbb{CP}^{2k}$ is a fiber bundle whose 
fibers are products of flag manifolds constructed from the tangent bundle
of the base. The associated distribution 
$\overline{\mathcal{D}}_{n,k}=(\pi_*)^{-1}(S'+\Sigma'')$, however, does possess 
singularities at all points where the sum $S'+\Sigma''$ is not direct.}
\end{remark}

\vskip 0,1cm
\noindent
{\em Symplectic case: Proof of Theorem~\ref{theorem-univ-embed-symplectic}.} Let $(X,J, \omega)$ be a compact $n$-dimensional almost complex symplectic manifold with second Betti number $b_2\le b$ and a $J$-compatible symplectic form $\omega$. We choose $b_2$ rational cohomology classes on $X$, denoted $[\omega_1],\dots,[\omega_{b_2}]$, that form a basis of the De Rham cohomology space $H^2(X,\mathbb R)$. For this, we take classes $[\omega_j] \in H^2(X,\mathbb Q)$ very close to $[\omega]$, such that $[\omega]$ lies in the interior of the simplex of vertices $[0]$, $[\omega_1],\,\ldots\,,[\omega_{b_2}]$. The 2-form $\omega_j$ can be taken to be very close to $\omega$ in uniform norm over~$X$. This ensures that the $\omega_j$'s are symplectic and that $[\omega]$ is a convex combination $[\omega] = \sum \lambda_j[\omega_j]$ for some $\lambda_1,\dots,\lambda_{b_2}>0$ with $0<\sum\lambda_j<1$ and $\sum\lambda_j\simeq 1$. Since $u=\omega - \sum\lambda_j\omega_j$ is a very small exact $2$-form $u$, we can in fact 
achieve $u=0$ after replacing one of the $\omega_j$'s by $\omega_j+\lambda_j^{-1}u$. Also, after replacing each $\omega_j$ by an integer multiple, we obtain $\omega=\sum \lambda_j\omega_j$ where $\omega_j$ is a system of integral symplectic forms and $\lambda_j>0$, $\sum\lambda_j<1$. After replacing $\omega_{b_2}$ by $b-b_2+1$ identical copies $\omega_j=\omega_{b_2}$, we can assume that $\omega=\sum_{1\le j\le b}\lambda_j\omega_j$, $\lambda_j>0$.

According to the effective version of Tischler's theorem stated by Gromov \cite{gromov} (page 335), for every $j=1,\dots,b$, there exists a symplectic embedding $g_j : (X,\omega_j) \rightarrow (\mathbb{CP}^k, \gamma_{\rm FS})$ with $k=2n+1$, where $\mathbb{CP}^k$ denotes the complex projective space of (complex) dimension $k$ and $\gamma_{\rm FS}$ denotes the Fubini-Study form on $\mathbb{CP}^k$. Then $g:=(g_1,\dots,g_{b})$ is a symplectic embedding of $(X,\omega)$ into the K\"ahler complex projective manifold 
$$(Y,\gamma_\lambda):=\bigg(\prod_{j=1}^{b} \mathbb{CP}^k~,~\sum_{j=1}^{b} \lambda_j\pr_j^*\gamma_{\rm FS}\bigg).$$
Here $\pr_j : \prod_{j=1}^{b} \mathbb{CP}^k \rightarrow \mathbb{CP}^k$ denotes the $j$-th projection. By construction, we have
$$
\omega=\sum_{j=1}^{b}\lambda_j\omega_j=g^*\gamma_\lambda.
$$
Let $NX$ be the normal bundle of $g(X)$ in $Y$. Here, we identify the normal bundle with a subbundle of $TY_{|g(X)}$ by using the symplectic structure, namely we define
$$NX=\big\{\eta\in TY\,;\,\forall\xi\in TX,~
\gamma_\lambda(g_*\xi,\eta)=0\big\};$$
the positivity condition $\gamma_\lambda(g_*\xi,g_*J_X\xi)=\omega(\xi,J_X\xi)>0$ for $\xi\ne 0$ implies that we indeed have $g_*TX\cap NX=\{0\}$, and thus $TY_{|g(X)}=g_*TX\oplus NX$. Although we will not make use of this, one can see that the Riemannian and symplectic normal bundles are linked by the relation $NX^{\rm riem}=J_{\rm st}NX^{\rm symp}$ where $J_{\rm st}$ is the standard complex structure of $Y$, the latter being unrelated to~$J_X$. We embed $X$ in $Y\times\overline{Y}$ by the ``diagonal'' embedding $x \mapsto G(x) = (g(x),\overline{g(x)})$ (set theoretically $\overline Y$ coincides with $Y$, but we take the conjugate complex structure $J_{\overline Y}=-J_Y$). We have a decomposition of the tangent bundle given, for $x\in X$, by
$$T(Y\times\overline Y)_{G(x)}=TX_x \oplus \overline{TX}_x \oplus NX_x \oplus \overline{NX}_x$$ 
where the first factor $TX_x$ consists of diagonal vectors $(g_*\xi,\overline{g_*\xi})$ (with the slight abuse of notation consisting in identifying $X$ and $g(X)\subset Y$), the second consists of ``antidiagonal'' vectors $(g_*\xi,-\overline{g_*\xi})$, and the two normal bundle copies are $\pr_1^*NX$ and $\pr_2^*\overline{NX}$. With respect to this decomposition, we define a complex structure $\tilde J(x)$ on $T(Y\times\overline Y)_{G(x)}$ by
$$
\tilde J(x) = J_X(x) \oplus (-J_X(x)) \oplus J_{NX \oplus \overline{NX}}(x)
$$
where $J_{NX \oplus \overline{NX}}$ is the tautological almost complex structure $(u,v) \mapsto (-\overline v,\overline u)$ on $NX \oplus\overline{NX}$. Clearly, we have $G^*\tilde J=J_X$ [in fact $Y\times \overline Y$ is just the complexification of the underlying real algebraic structure $Y^{\mathbb R}$ on $Y$, under the anti-holomorphic involution $(x,y)\mapsto (y,x)$]. Let us consider the K\"ahler structure 
$$\tilde\gamma=\frac{1}{2}\big(\pr_1^*\gamma_\lambda-\pr_2^*\gamma_\lambda\big)\quad\hbox{on $Y\times\overline Y$}.$$
Notice that $-\gamma_\lambda$ is in fact a K\"ahler structure on $\overline Y$ and that $\omega=g^*\gamma_\lambda={\overline g}^*(-\gamma_\lambda)$.
We thus have $G^*\tilde\gamma=g^*\gamma_\lambda= \omega$, and further claim
that $\tilde J$ is compatible with $\tilde\gamma$. In order to check this, let us take two tangent vectors $(\xi_1,\overline\xi_2),\;(\eta_1,\overline\eta_2)\in TY\times T\overline Y$, and write $\xi=\xi'+\xi''$ for the decomposition of $\xi\in TY$ along $TX\oplus NX$. We find
\begin{eqnarray*}
\tilde J(\xi_1,\overline\xi_2)&=&\tilde J\left(\frac{1}{2}(\xi'_1+\xi'_2,\overline\xi'_1+\overline\xi'_2)+\frac{1}{2}(\xi'_1-\xi'_2,\overline\xi'_2-\overline\xi'_1)+(\xi_1'',0)+(0,\overline\xi''_2)\right)\\
&=&\frac{1}{2}\left(J_X(\xi'_1+\xi'_2),-\overline{J_X(\xi'_1+\xi'_2)}\right)+\frac{1}{2}\left(-J_X(\xi'_1-\xi'_2),\overline{J_X(\xi'_2-\xi'_1)}\right)\\
&&{}+(-\xi_2'',0)+(0,\overline\xi''_1)\\
\noalign{\vskip6pt}
&=&\left(J_X\xi'_2-\xi_2'',\overline{-J_X\xi'_1+\xi''_1}\,\right).
\end{eqnarray*}
Since $J_X$ and $\omega$ are compatible and $\omega=g^*\gamma_\lambda$, we have
(with our abuse of notation $\xi'_1\simeq g_*\xi'_1$)
$$
\gamma_\lambda(J_X\xi'_1,J_X\eta'_1)=
\omega(J_X\xi'_1,J_X\eta'_1)=\omega(\xi'_1,\eta'_1)=\gamma_\lambda(\xi'_1,\eta'_1)
$$
and a similar formula for $(\xi'_2,\eta'_2)$. We infer from this and from
the $\gamma_\lambda$-orthogonality of the decomposition $TX\oplus NX$ that
\begin{eqnarray*}
\tilde\gamma\left(\tilde J(\xi_1,\overline\xi_2),
\tilde J(\eta_1,\overline\eta_2)\right)
&=&\frac{1}{2}
\gamma_\lambda\left(J_X\xi'_2-\xi_2'',J_X\eta'_2-\eta_2''\right)-\frac{1}{2}
\gamma_\lambda\left(\overline{-J_X\xi'_1+\xi''_1}\,,\,
\overline{-J_X\eta'_1+\eta''_1}
\right)\\
&=&\frac{1}{2}\gamma_\lambda\left(J_X\xi'_2-\xi_2'',J_X\eta'_2-\eta_2''\right)+
\frac{1}{2}
\gamma_\lambda\left(-J_X\xi'_1+\xi''_1,-J_X\eta'_1+\eta''_1\right)\\
&=&\frac{1}{2}\gamma_\lambda\left(\xi'_2,\eta'_2\right)
+\frac{1}{2}\gamma_\lambda\left(\xi''_2,\eta_2''\right)+
\frac{1}{2}\gamma_\lambda\left(\xi'_1,\eta'_1\right)+
\frac{1}{2}\gamma_\lambda\left(\xi''_1,\eta''_1\right)\\
&=&\frac{1}{2}\gamma_\lambda\left(\xi_1,\eta_1\right)
-\frac{1}{2}\gamma_\lambda\left(\overline\xi_2,\overline\eta_2\right)\\
\noalign{\vskip3pt}
&=&\tilde\gamma\left((\xi_1,\overline\xi_2),(\eta_1,\overline\eta_2)\right).
\end{eqnarray*}
This proves that $\tilde J$ is compatible with the restriction of the K\"ahler structure $\tilde\gamma$ to $G(X)$.

We now construct $Z_{n,b,k}$, following essentially the same lines as for the proof of Theorem~\ref{theorem-univ-embed}. We view $\tilde Y=Y\times \overline Y$ as a real algebraic manifold equipped with a real algebraic symplectic form $\tilde\gamma$ (though $\tilde Y$ is in fact complex projective and $\tilde\gamma$ K\"ahler). We consider
$$W =\left\{(w,S,J) \in {\tilde Y} \times G_{\mathbb R}(T{\tilde Y},2n) \times 
\End_{\mathbb R}(T{\tilde Y})\,;\;J^2 = -I,\ J^*\tilde\gamma=\tilde\gamma,\ J(S) \subset S\right\}$$
and its complexification $W^{\mathbb C}$ which is defined by the same algebraic equations over $\mathbb C$. We define $Z_{n,b,k}$ to be the component of $W^{\mathbb C}$ for which
$J$ and $J_{|S}$ have balanced $+i$ and $-i$ eigenspaces $\Sigma'\oplus\Sigma''=T\tilde Y^{\mathbb C}$ and $S'\oplus S''=S$. There is a natural projection
$$\pi=\pi_{n,b,k}:Z_{n,b,k}\to \tilde Y^{\mathbb C}=Y^2\times \overline Y^2,$$
and $\tilde\gamma$ can be complexified into a K\"ahler form $\tilde\gamma^{\mathbb C}$ on $\tilde Y^{\mathbb C}$ which restricts to $\tilde\gamma$ on the real~part. Our construction produces a canonical lifting $f:X\to Z_{n,b,k}$ of $G:X\to \tilde Y=\tilde Y^{\mathbb R}\subset\tilde Y^{\mathbb C}$. Then, as above, the bundle ${\mathcal D}_{n,b,k}$ of tangent vectors $\zeta\in TZ_{n,b,k}$ such that \hbox{$\pi_*\zeta\in S'\oplus\Sigma''$} defines an algebraic distribution transverse to $G(X)$, and additionally $\beta=\pi^*\tilde\gamma^{\mathbb C}$ is a transverse K\"ahler form that induces the given symplectic structure $\omega$ on $X$. A calculation of dimensions shows that $\dim_{\mathbb C}\tilde Y^{\mathbb C}=4bk$ and
$\dim_{\mathbb C}\tilde Z_{n,b,k}=2bk(2bk+1)+2n(2bk-n)$, since the symplectic twistor space $\{J\}$ in dimension $2m=4bk$ has dimension $m(m-1)$, and we have additionally to select $n$-dimensional subspaces $S'$, $S''$ in the given $2bk$ dimensional eigenspaces of $J$. The above variety $Z_{n,b,k}$ is merely quasi-projective, but we can of course replace it with a relative projective compactification over $\smash{\tilde Y^{\mathbb C}}$, and extend  ${\mathcal D}_{n,b,k}$ as a torsion free algebraic subsheaf of $TZ_{n,b,k}$.\qed

\begin{remark}\label{genuine-kahler-remark}{\rm In the above result, we could even take $\beta$ to be a genuine K\"ahler metric on $Z_{n,b,k}$. In fact $Z_{n,b,k}$ is (quasi-)projective, so it possesses a K\"ahler metric $\gamma$. One can easily conclude by a perturbation argument, after replacing $\beta$ by $\beta+\varepsilon\gamma$ and letting $\omega$ vary  in a neighborhood of the original symplectic form on~$X$.}
\end{remark}

\section{A weak version of Bogomolov's conjecture: proof of Theorem~\ref{bogomolov-thm}}\label{section-bogomolov}
We start with a formula computing the Nijenhuis tensor of the almost complex structure $J_f$ given by an embedding $f:X\hookrightarrow Z$ transverse to a holomorphic distribution $\mathcal D$. Recall that for any smooth (real) vector fields $\zeta$, $\eta$ of $TX$, the Nijenhuis tensor $N_J$ of an almost complex structure $J$ is defined in terms of Lie brackets of $\zeta^{0,1}=\frac{1}{2}(\zeta+iJ\zeta)$, $\eta^{0,1}=\frac{1}{2}(\eta+iJ\eta)$ as
$$
N_J(\zeta,\eta)=4\Re\,[\zeta^{0,1},\eta^{0,1}]^{1,0}=[\zeta,\eta]-[J\zeta,J\eta]
+J[\zeta,J\eta]+J[J\zeta,\eta].
$$

\begin{proposition}\label{integrable-prop}If $\theta$ denotes the torsion operator of the distribution $\mathcal D$ on $Z$, the Nijenhuis tensor of the almost complex structure $J_f$ induced by a transverse embedding $f:X\hookrightarrow Z$ is given by
\begin{equation}\label{nij-eq}
\forall z \in X,\ \forall \zeta,\eta \in T_zX,\quad N_{J_f}(\zeta,\eta) = 4\,\theta(\overline{\partial}_{J_f}f(z)\cdot \zeta, \overline{\partial}_{J_f}f(z)\cdot \eta).
\end{equation}
\end{proposition}

\begin{proof} We keep the same notation as in Section~\ref{deform-section}. Especially, we put $M=f(X)$, and near any point $x_0 \in X$, we write $f = F \circ \varphi$ where $\varphi$ is a local diffeomorphism defined in a neighborhood $V$ of $x_0$ and $F : \varphi \ni z' \mapsto (z',g(z')) \in Z$. According to (\ref{formula-JF}), the almost complex structure $J_F$ on $\varphi(V)\subset M$ and the corresponding one $J_f$ on $V\subset X$ are given by
$$
\forall z' \in \varphi(V),\ \forall \zeta \in T_zM,\ J_F(z) \cdot \zeta = i\zeta -2dF(z')^{-1}\pi_{Z,\mathcal D,M}(ia(F(z'))(\overline{\partial} g(z')\cdot \zeta),0)
$$
and  $J_f(x) = d\varphi(x)^{-1} \circ J_F(\varphi(x)) \circ d\varphi(x)$ for every $x \in V$. Thus, by construction, we have $\overline{\partial}_J f= \overline{\partial}_{J_F} F \circ d\varphi$ (i.e.\ $d\varphi$ is compatible with $J_f$ and $J_F$), and (\ref{nij-eq}) is equivalent to
$$
\forall z \in M,\ \forall \zeta,\eta \in T_{z}M,\ N_{J_F}(\zeta,\eta) = 4\,\theta(\overline{\partial}_{J_F}F(z)\cdot \zeta, \overline{\partial}_{J_F}F(z)\cdot \eta).
$$
Let $\zeta = \sum_{j=1}^n \left(\zeta_j \frac{\partial}{\partial z_j}+ \overline{\zeta_j} \frac{\partial}{\partial \overline{\zeta}_j}\right)$,
$\eta = \sum_{j=1}^n \left(\eta_j \frac{\partial}{\partial z_j}+ \overline{\eta_j} \frac{\partial}{\partial \overline{\eta}_j}\right)$ 
be real vector fields in $z'\mapsto T_{z'}M$. For the sake of clarity we denote by $J_{\rm st}\zeta$ the vector field $i\zeta$ associated with the ``standard'' almost complex structure of $\mathbb{C}^n$, so that
$$J_{\rm st}\zeta=\sum_{j=1}^n\left(i\zeta_j \frac{\partial}{\partial z_j}-i\overline{\zeta_j} \frac{\partial}{\partial \overline{\zeta}_j}\right).$$
Without loss of generality (and in order to simplify calculations), we assume $\zeta,\,\eta$ to have constant coefficients $\zeta_j,\,\eta_j$. At the central point $z'_0 \in \varphi(V)$ we have $a(F(z'_0))=0$ and $J_F=J_{\rm st}$, hence (omitting vanishing terms such as $[\zeta,\eta]$), the Nijenhuis tensor of $J_F$ at~$z'_0$ is given by
$$
\begin{array}{lll}
N_{J_F}(\zeta,\eta)_{|z'_0} \kern-5pt& = \kern-5pt&{}
-J_{\rm st}\zeta\cdot \left(-2dF(z')^{-1} \pi_{Z,\mathcal D,M}(ia(F(z'))\cdot (\overline{\partial}g(z')\cdot \eta),0)\right)_{|z'=z'_0}\\
 \noalign{\vskip5pt}
 & &{}+J_{\rm st}\eta\cdot \left(-2dF(z')^{-1} \pi_{Z,\mathcal D,M}(ia(F(z'))\cdot (\overline{\partial}g(z')\cdot \zeta),0)\right)_{|z'=z'_0}\\
 \noalign{\vskip5pt}
 & &{} +J_{\rm st}\left[\zeta\cdot \left(-2dF(z')^{-1} \pi_{Z,\mathcal D,M}(ia(F(z'))\cdot (\overline{\partial}g(z')\cdot \eta),0)\right)\right]_{|z'=z'_0}\\
 \noalign{\vskip5pt}
 & &{} +J_{\rm st}\left[-\eta\cdot \left(-2dF(z')^{-1} \pi_{Z,\mathcal D,M}(ia(F(z'))\cdot (\overline{\partial}g(z')\cdot \zeta),0)\right)\right]_{|z'=z'_0}.
\end{array}
$$
We recall that by our normalization $a(F(z'_0)) = 0$ and $dF(z'_0) \circ \pi_{Z,\mathcal D,M}(z'_0) = {\rm id}$. Hence for all $\zeta,\,\eta \in T_{z'_0}M$ we get at $z'_0$
$$
\begin{array}{lll}
 N_{J_F}(\zeta,\eta) & = & \left(2ida(z'_0)(dF(z'_0)\cdot J_{\rm st}\zeta)\cdot (\overline{\partial}g(z'_0)\cdot \eta),0\right)\\
 \noalign{\vskip7pt}
 & &{} +J_{\rm st}\left(-2ida(z'_0)(dF(z'_0)\cdot \zeta)\cdot \overline{\partial}g(z'_0)\cdot \eta,0\right)\\
  \noalign{\vskip7pt}
 & &{} - \left(2ida(z'_0)(dF(z'_0)\cdot J_{\rm st}\eta)\cdot \overline{\partial}g(z'_0)\cdot \zeta,0\right)\\
  \noalign{\vskip7pt}
 & &{} - J_{\rm st}\left(-2ida(z'_0)(dF(z'_0)\cdot \eta)\cdot \overline{\partial}g(z'_0)\cdot \zeta,0\right).\\
\end{array}
$$
Since $J_{\rm st} \circ da(F(z'_0)) \circ J_{\rm st} = -da(F(z'_0))$, we infer
\begin{eqnarray*}
&&\left(da(z'_0)(dF(z'_0)\cdot J_{\rm st}\zeta)\cdot (\overline{\partial}g(z'_0)\cdot \eta)-J_{\rm st}(da(z'_0)(dF(z'_0)\cdot \zeta)) \cdot \overline{\partial}g(z'_0)\cdot \eta,0\right)\\
 \noalign{\vskip5pt}
&&\kern50pt{}= -\left(J_{\rm st}\left(da(F(z'_0))\left(dF(z'_0)+J_{\rm st} \circ dF(z'_0) \circ J_{\rm st}\right)\cdot \zeta\right)\cdot (\overline{\partial}g(z'_0)\cdot \eta),0\right).
\end{eqnarray*}
Finally, since $dF(z'_0)+J_{\rm st} \circ dF(z'_0) \circ J_{\rm st} = 2\overline{\partial}F(z'_0) = 2\overline{\partial}g(z'_0)$ we obtain for every $\zeta,\,\eta \in T_{z'_0}M$
$$
\begin{array}{lll}
N_{J_F}(\zeta,\eta) & = &{} -4i\left(J_{\rm st}\left(da(F(z'_0))(\overline{\partial}_{J_F}F(z'_0)\cdot \zeta)\cdot (\overline{\partial}_{J_F}F(z'_0)\cdot \eta)\right),0\right) \\
\noalign{\vskip7pt}
& &{} +4i\left(J_{\rm st}\left(da(F(z'_0))(\overline{\partial}_{J_F}F(z'_0)\cdot \eta)\cdot (\overline{\partial}_{J_F}F(z'_0)\cdot \zeta)\right),0\right)\\
\noalign{\vskip7pt}
 & = &{}-4i\,J_{\rm st}\, \theta\left(\overline{\partial}_{J_F}F(z'_0)\cdot \zeta, \overline{\partial}_{J_F}F(z'_0)\cdot \eta \right),
\end{array}
$$
which yields the expected formula.\qed
\vskip 0.3cm

\noindent{\em Proof of Theorem~\ref{bogomolov-thm}.}
Let $(X,J)$ be a complex manifold and let $f:X\hookrightarrow Z_{n,k}$ be an embedding into the universal space $(Z_{n,k},{\mathcal D}_{n,k})$, such that $J_f=J$. Then $J_f$ is an integrable structure, hence $N_{J_f}$ vanishes identically. It follows from (\ref{nij-eq}) that $\Im(\overline{\partial}_{J}f)$ is contained in the isotropic locus of $\theta$, namely the set of $n$-dimensional subspaces $S$ in the Grassmannian bundle ${\rm Gr}({\mathcal D}_{n,k},n)\to Z_{n,k}$ such that $\theta_{z|S\times S}=0$ at any point $z\in Z_{n,k}$.
\end{proof}

\section{Relation to Nash Algebraic approximations of holomorphic foliations\label{nash-approx-foliations}} We prove in this section Proposition~\ref{runge-prop}. Let $X$ be a compact complex manifold. We denote by $J_X$ the corresponding (almost) complex structure. We first embed $X$ diagonally into $X \times \overline{X}$. This is a totally real embedding with normal bundle $T\overline{X}$. 
If we denote by $\varphi$ the embedding then $\varphi(X)$ is totally real and compact in $X \times \overline{X}$. Moreover if $\pr_1 : X \times \overline{X} \to X$ is the projection on the first factor then $\Ker(d\pr_1)=\pr_2^*T\overline X$ defines a holomorphic foliation of $X \times \overline{X}$ and, quite trivially, $\varphi(X)$ is transverse to $\Ker(d\pr_1)$.

Since $X\times\overline X$ is a complexification of the real analytic manifold $\varphi(X)$, a well known result of Grauert \cite{grauert}) shows that $\varphi(X)$ possesses a Stein tubular neighborhood $U$ in $X\times\overline X$ (by Nirenberg and Wells \cite{nir-wells}, every totally real submanifold of a complex manifold has in fact a fundamental system of Stein neighborhoods, and in the compact case they can be obtained as tubes $U_\varepsilon=\{d(z,w)<\varepsilon\}$ for the geodesic distance associated with any hermitian metric on~$X$). According to a result of E.L.~Stout \cite{stout}, the Stein neighborhood $U$ can be shrunk to a Runge open subset $U'\Subset U$ that is biholomorphic to a bounded polynomial polyhedron $\Omega$ in an affine complex algebraic manifold $Z=\{P_j(z)=0\}\subset{\mathbb C}^N$, say $\Omega=\{z\in Z\,;\;|Q_j(z)|<1\}\Subset Z$, where $P_j,\,Q_j\in{\mathbb C}[z_1,\ldots,z_N]$. Let \hbox{$\psi:U'\to \Omega$} be this biholomorphism, let $f=\psi\circ\varphi:X\hookrightarrow Z$ be the 
resulting real analytic embedding and let $J_Z$ be the complex structure on $Z$. We denote by $M=f(X)=\psi(\varphi(X))$ the image of $X$ by $f$ and by $\mathcal F=f_*(\Ker(d\pr_1)_{|U'})$ the direct image of $\Ker(d\pr_1)$ restricted to~$U'$. Then $\mathcal F\subset TZ_{|\Omega}$ is a holomorphic foliation of codimension $n$ on~$\Omega$, and $M$ is transverse to $\mathcal F$. By construction, the complex structure $J_M^{Z,\mathcal F}$ on $M$ induced by $(TZ/\mathcal F,J_Z)$ on $M$ coincides with~$f_*(J_X)$. Now, we invoke the following 
\begin{proposition}\label{affine-thickening}There exists a Runge open subset $\tilde\Omega\subset{\mathbb C}^N$ such that $\Omega=\tilde\Omega\cap Z$, a holomorphic retraction $\rho:\tilde\Omega\to\Omega$, and a holomorphic foliation $ \tilde{\mathcal F}$ of codimension $n$ on $\tilde\Omega$ such that
$M\subset\Omega$ is transverse to $\tilde{\mathcal F}$.
\end{proposition}
\begin{proof}The normal bundle sequence
$$
0\to TZ\to T{\mathbb C}^N_{|Z}\to NZ\to 0
$$
admits an algebraic splitting $\sigma:NZ\to T{\mathbb C}^N_{|Z}$ since $H^1_{\rm alg}(Z,\Hom(NZ,TZ))=0$ ($Z$ being affine). Then $h(z,\zeta)=z+\sigma(z)\cdot \zeta$, $\zeta\in NZ_z$, defines an algebraic biholomorphism $h$ from a tubular neighborhood $V$ of the zero section of $NZ$ onto a neighborhood $h(V)$ of $Z$ in ${\mathbb C}^N$. Clearly, if $\pi:NZ\to Z$ is the natural projection, $\rho=\pi\circ h^{-1}$ is a Nash algebraic retraction from $\tilde V:=h(V)$ onto $Z$. We take 
$$
\tilde\Omega=\left\{z\in{\mathbb C}^N\,;\;|P(z)|^2:=\sum|P_j(z)|^2<\varepsilon(1+|z|^2)^{-A},\;|Q_j(z)|^2+C|P(z)|<1\right\}$$
with $\varepsilon\ll 1$ and $A,\,C\gg 1$ chosen so large that $\tilde\Omega\Subset h(V)$. Then $\rho$ maps $\tilde\Omega$ submersively onto~$\Omega$, and $\tilde\Omega$~is a Runge open subset in ${\mathbb C}^N$. We simply take ${\tilde{\mathcal F}}=(\rho_*)^{-1}{\mathcal F}\subset T\tilde\Omega$ to be the inverse image of $\mathcal F$ in $\tilde\Omega$.
\end{proof}
\vskip2mm

\noindent
{\em End of proof of Proposition~\ref{runge-prop}.} Thanks to Prop.~\ref{affine-thickening} and our preliminary discussion, we may assume that $f:X\hookrightarrow Z\cap\tilde\Omega\subset\tilde\Omega$ is a real analytic embedding into a Runge open subset $\tilde\Omega\subset{\mathbb C}^N$, transversally to a holomorphic foliation $\tilde{\mathcal F}$ on $\tilde\Omega$, with $J_X=J_f$. Assume that such foliations can be approximated by Nash algebraic foliations $\tilde{\mathcal F}_\nu$ on~$\tilde\Omega$, uniformly on compact subsets. This means that $\tilde{\mathcal F}_\nu$ is given by a Nash algebraic distribution $\delta_\nu:\tilde\Omega\to{\rm Gr}({\mathbb C}^N,n)$ that is moreover integrable. It is worth observing that if the integrability assumption is dropped, then the existence of the Nash algebraic approximating sequence $\delta_\nu$ is actually granted by \cite{DLS93} (but it seems quite difficult to enforce the integrability condition in this context).

Now $M=f(X)$ is still transverse to ${\mathcal F}_\nu$ for $\nu\ge \nu_0$, and in this way we would obtain a sequence of integrable complex structures $J_\nu=f^*J^{Z,{\mathcal F}_\nu}_M$ on $X$ that approximate $J_X=f^*J^{Z,{\mathcal F}}_M$ in the Kuranishi space of small deformations of~$X$.\qed

\begin{remark} {\rm Instead of embedding just $X$ in an affine algebraic manifold~$Z$, we could instead embed the whole Kuranishi space ${\mathcal X}\to S$ of $X$ into a product $Z\times S$, keeping $S$ as a parameter space (it is enough to take a small Stein neighborhood $S'\subset S$ containing the base point $0$ of the central fiber ${\mathcal X}_0=X$). For this we simply observe that $X$ embeds diagonally as a totally real submanifold $\Delta\subset X\times\overline X\subset {\mathcal X}\times \overline X$, hence $\Delta$ admits a Stein neighborhood as before, and we can embed the latter as a Runge open set in an affine algebraic manifold. By looking at the projection $\pr_1$ to $\mathcal{X}$ and at the fibers over $S$, the vertical tangent bundle $\pr_2^*T\overline X$ defines a holomorphic foliation transverse to all complex structures $\mathcal{X}_t$ close to
$\mathcal{X}_0=X$. It is then not unlikely that one could also embed the original complex structure $J_X$ (without having to take an approximation), by using some sort of openness argument and the fact that we have embedded the whole Kuranishi space.}
\end{remark}

\begin{remark} {\rm It should also be observed that there is probably no topological obstruction to the approximation problem. In fact, it is enough to consider the case where we have a real analytic embedding $f:X\hookrightarrow Z$, transversal to a holomorphic foliation given by a {\em trivial} subbundle ${\mathcal F}\subset TZ$ on some Runge open set $\Omega\Subset Z$ of an affine algebraic manifold. Otherwise, thanks to \cite{DLS93}, one can always assume that $\mathcal F$ is isomorphic to the restriction of an algebraic vector bundle $\mathcal F'$ on $Z$ (possibly after shrinking~$\Omega$ and replacing $Z$ by some finite cover $Z'\to Z$). Then, since $Z$ is affine, one can find a surjective algebraic morphism \hbox{$\mu:{\mathcal O}_Z^{\oplus p}\to{\mathcal F'}$}. Its kernel ${\mathcal G}=\Ker\mu$ satisfies ${\mathcal F}'\oplus{\mathcal G}\simeq{\mathcal O}_Z^{\oplus p}$, i.e.\ ${\mathcal F}'\oplus{\mathcal G}$ is algebraically trivial. We replace $Z$ by $\tilde Z={\mathcal G}$ (the total space of $\mathcal G$), and $\mathcal F$ by the inverse image \hbox{$\tilde{\mathcal F}=(\pi_*)^{-1}({\mathcal F})\subset T\tilde Z$} via $\pi:\tilde Z\to Z$. Then $\tilde{\mathcal F}\simeq\pi^*({\mathcal F}'\oplus{\mathcal G})$, hence $\tilde{\mathcal F}$ is holomorphically trivial. We can therefore always reduce ourselves to the case where $\mathcal F$ is holomorphically trivial. When this is the case, $\mathcal F$ admits a global holomorphic frame $(\zeta_j)$, and the Lie brackets satisfy $[\zeta_j,\zeta_k]=\sum_\ell u_{jk\ell}\zeta_\ell$ for some uniquely defined holomorphic functions $u_{jk\ell}$ on $\Omega$. These functions can of course be approximated by a sequence of polynomials $p^\nu_{jk\ell}\in{\mathbb C}[Z]$, $\nu\in{\mathbb N}$, but it is unclear how to construct Nash algebraic foliations from these data$\,\ldots$}
\end{remark}

\section{Categorical viewpoint\label{categorical}} 
Compact complex manifolds form a natural category CCM, where the morphisms are 
by defini\-tion holomorphic maps $\Psi:X\to Y$. The present work is also
concerned with the 
category FAV of ``foliated algebraic varieties'' for which 
the objects are suitable triples~$(Z,\mathcal{D},f)$.
Here, $Z$ is a non singular complex algebraic variety, 
\hbox{$\mathcal{D}\subset\mathcal{O}(TZ)$}
an integrable algebraic subsheaf and $f:X\to Z$ a
$\mathcal{C}^\infty$ embedding of a compact even 
dimensional real manifold~$X$ transversally to $\mathcal{D}$. A~morphism
$$
(Z,\mathcal{D},f)\longrightarrow (W,\mathcal{E},g)\qquad
\hbox{(where $g:Y\to W$ is transverse to $\mathcal{E}$),}
$$
is a pair $(\Psi,\varphi)$ where $\Psi:Z\to W$ is an algebraic morphism 
such that $d\Psi(\mathcal{D})\subset\Psi^*\mathcal{E}$, and 
$\varphi:X\to Y$ a differentiable map with the property that
$\Psi\circ f=g\circ\varphi$. By definition, $\varphi$~is then a holomorphic
map from $(X,J_f)$ into $(Y,J_g)$, since $d\Psi$ induces a holomorphic
(and even algebraic) morphism
$$
d\Psi\mathop{\rm mod}\mathcal{D}:TZ/\mathcal{D}\to\Psi^*(TW/\mathcal{E})
$$
between the ``transverse structures''. We have a natural fonctor
$$
{\rm FAV}\to{\rm CCM},\qquad 
(Z,\mathcal{D},f)\mapsto (X,J_f)\label{bogo-fonctor}
$$
and the basic question~\ref{bogo-question} can then be reformulated

\begin{question}\label{is-fonctor-surjective}
Is the fonctor ${\rm FAV}\to{\rm CCM}$ surjective~?
\end{question}

Here, one can of course identify triples in FAV given by isotopy equivalent
transverse embeddings $f:X\to Z$, and complex structures $(X,J)$, $(X,J')$
that are equivalent through a path of biholomorphic maps
$h_t:(X,J)\to(X,J_t)$, $t\in[0,1]$, $J_0=J$, $J_1=J'$, $h_0=\Id_X$. One could 
also say that a morphism
$$
(\Psi,\varphi):(Z,\mathcal{D},f:X \to Z)\longrightarrow
(W,\mathcal{E},g:Y\to W)
$$
defines an ``isomorphism of transverse structures'' if $\varphi:X\to Y$ is a 
diffeomorphism and
$$
d\Psi\mathop{mod}\mathcal{D}:TZ/\mathcal{D}\to\Psi^*(TW/\mathcal{E})
$$
yields a bundle isomorphism in restriction to $f(X)$ (so that d$\Psi
\mathop{\rm mod}\mathcal{D}$ has to be a generic isomorphism between 
the quotients;
one could further allow $\Psi$ to be merely a rational morphism provided the indeterminacy set does not intersect $f(X)$). Such an isomorphism of transverse structure is obtained by taking $\Psi=\pr_1:Z\times A\to Z$ with an arbitrary algebraic variety $A$, pulling-pack $\mathcal{D}$ to $T(Z\times A)$ and replacing
$f$ with $f\times\{a_0\}:X\to Z\times\{a_0\}\,$; in this way one can add many ``extra parameters'' to $Z$ that play no role at all in the transverse complex structure. One can obtain another such situation by blowing up or blowing down subvarieties of $Z$ that do not intersect $f(X)$, possibly after displacing $f(X)$ by a transverse isotopy. Question~\ref{is-fonctor-surjective} can then be completed as follows, but we currently have extremely little evidence about it:

\begin{question}\label{what-are-fibers}
What are the fibers of ${\rm FAV}\to{\rm CCM}$, at least if we identify
objects through transverse isotopies and isomorphisms of transverse 
structures~? Are they parametrized by finite dimensional moduli spaces~?
\end{question}

Although we are not able to answer these questions, it can be observed that
our weak version of Bogomolov's conjecture (Theorem~\ref{bogomolov-thm})
also admits a categorical interpretation. For this, we introduce a
category WFAV of ``weakly foliated algebraic varieties''. Its objects 
are quadruplets
$$
(Z,\mathcal{D},\mathcal{S},f)
$$
where $Z$ is a complex algebraic variety, $\mathcal{S}\subset\mathcal{D}$
are nested algebraic subsheaves of $\mathcal{O}(TZ)$ such that $[\mathcal{S},
\mathcal{S}]\subset\mathcal{D}$, and $f:X\to Z$ is a smooth embedding
of a compact even dimensional real manifold $X$ into $Z$ such that $f(X)$ 
is transverse to $\mathcal{D}$ and
$\Im(\overline\partial_{J_f}f)\subset\mathcal{S}$ for the induced
almost complex structure~$J_f$. It then follows from 
Prop.~\ref{integrable-prop} that $(X,J_f)$ is still integrable, thus
$(X,J_f)$ is a compact complex manifold. One can see that a transverse
isotopy $(f_t)_{t\in[0,1]}$ such that 
$\Im(\smash{\overline\partial}f_t)\subset\mathcal{S}$ yields
a path of biholomorphic structures. Morphisms
$$
(\Psi,\varphi):(Z,\mathcal{D},\mathcal{S},f)\to
(W,\mathcal{E},\mathcal{T},g)
$$
are defined to be algebraic morphisms $\Psi:Z\to W$ such that $d\Psi$ maps
$\mathcal{D}$ in $\mathcal{E}$, $\mathcal{S}$ in $\mathcal{T}$, and
$\Psi\circ f=g\circ\varphi$ with $\varphi:X\to Y$. We then get the following 
affirmative answer to the weak analogue of Bogomolov's question.

\begin{theorem}The natural fonctor ${\rm WFAV}\to{\rm CCM}$ defined by
$(Z,\mathcal{D},\mathcal{S},f)\mapsto(X,J_f)$ is surjective.
\end{theorem}

\begin{proof} First, we apply Theorem~\ref{bogomolov-thm} to a compact complex manifold $(X,J)$ to obtain a smooth embedding $f:X\to Z_{n,k}$ that is transverse to a distribution $\mathcal{D}_{n,k}\subset TZ_{n,k}$. We replace $Z_{n,k}$ by $W_{n,k}\subset{\rm Gr}(\mathcal{D}_{n,k},n)$, where $W_{n,k}$ is defined as the set of all $n$-dimensional subspaces
$S\in{\rm Gr}(\mathcal{D}_{n,k},n)$ such that $\theta_{|S\times S}=0$. Notice that $W_{n,k}\to Z_{n,k}$ is actually a smooth fiber bundle, thanks to the homogeneous action of  the linear group preserving the flag construction. Let $\pi:W_{n,k}\to Z_{n,k}$. We equip $W_{n,k}$ with the distribution $\mathcal{E}_{n,k}:=(d\pi)^{-1}(\mathcal{D}_{n,k})$. Then $\mathcal{E}_{n,k}$ possesses a natural rank $n$ subbundle $\mathcal{T}_{n,k}\subset\mathcal{E}_{n,k}$ defined as the restriction of the tautological subbundle on the Grassmannian bundle. Now, the smooth map
$g=(f,\overline\partial f)$ defines a morphism $g:X\to W_{n,k}$ that is 
transverse to $\mathcal{E}_{n,k}$, and we have
$[\mathcal{E}_{n,k},\mathcal{E}_{n,k}]\subset \mathcal{T}_{n,k}$ by construction.
It is easy to see that the object $(W_{n,k},\mathcal{E}_{n,k},
\mathcal{T}_{n,k},g)\in{\rm WFAV}$ is mapped to $(X,J)\in{\rm CCM}$ by the
natural fonctor.
\end{proof} 
\vskip8mm

\bibliographystyle{amsalpha}

\vskip1cm\noindent
Jean-Pierre Demailly \&\ Herv\'e Gaussier\\
Universit\'e Grenoble Alpes, Institut Fourier,\\
UMR 5582 du CNRS, 100 rue des Maths\\
38610 Gi\`eres, France\\
\emph{e-mail}\/: jean-pierre.demailly@univ-grenoble-alpes.fr,\\
\phantom{\emph{e-mail}\/:} herve.gaussier@univ-grenoble-alpes.fr

\end{document}